\documentclass[reqno,11pt]{amsart}

\usepackage{amsmath,a4wide}
\usepackage{amsfonts}
\usepackage{verbatim}
\usepackage{mathtools}
\usepackage{amsthm}
\usepackage{indentfirst}
\usepackage{amsmath,amssymb}
\usepackage{amsthm}
\usepackage{fancyhdr}
\usepackage{mathrsfs}
\usepackage[hidelinks]{hyperref}
\usepackage{tikz}
\usepackage{subfig}
\usepackage{pgfplots}
\usepgfplotslibrary{polar}
\usepackage{caption}
\usepackage{graphicx}
\usepackage{enumerate}
\usepackage{mathrsfs}
\usepackage{amssymb}
\usepackage{amsthm}
\usepackage{amsmath,amsfonts,amssymb,esint}
\usepackage{graphics,color}
\usepackage{enumerate}

\usepackage{marginnote}
\usepackage{xfrac}
\usepackage{mathtools}
\usepackage{hyperref}

\usepackage{comment}

\DeclarePairedDelimiter{\abs}{\lvert}{\rvert}
\DeclarePairedDelimiter{\norm}{\lVert}{\rVert}

\graphicspath{{Immagini//}}

\newcommand{\T}{\ensuremath{\mathbb{T}}}
\newcommand*{\R}{\ensuremath{\mathbb{R}}}

\newcommand*{\N}{\ensuremath{\mathbb{N}}}
\newcommand*{\Z}{\ensuremath{\mathbb{Z}}}
\newcommand*{\C}{\ensuremath{\mathbb{C}}}

\renewcommand*{\div}{\ensuremath{\mathrm{div\,}}}
\newcommand*{\tr}{\ensuremath{\mathrm{tr\,}}}

\newcommand*{\Id}{\ensuremath{\mathrm{Id}}}

\newcommand{\eps}{\varepsilon}

\newcommand*{\curl}{\ensuremath{\mathrm{curl\,}}}

\renewcommand*{\P}{\ensuremath{\mathcal{P}}}

\newcommand*{\RR}{\ensuremath{\mathcal{R}}}

\newcommand{\vertiii}[1]{{\left\vert\kern-0.25ex\left\vert\kern-0.25ex\left\vert #1 
    \right\vert\kern-0.25ex\right\vert\kern-0.25ex\right\vert}}

\newtheorem{theorem}{Theorem}[section]
\newtheorem{lemma}[theorem]{Lemma}
\newtheorem{proposition}[theorem]{Proposition}
\newtheorem{corollary}[theorem]{Corollary}
\newtheorem{definition}[theorem]{Definition\rm}
\newtheorem{remark}{Remark}

\begin{document}
\pagenumbering{arabic}
\title[Ill-posedness for ipodissipative Navier--Stokes]{Ill-posedness of Leray solutions for the ipodissipative Navier--Stokes equations}

\author[Colombo]{Maria Colombo}
\address{Institute for Theoretical Studies, ETH Z\"urich,
\\ 
and Institut f\"ur Mathematik, Universit\"at Z\"urich, CH-8057 Z\"urich}
\email{maria.colombo@math.uzh.ch}

\author[De Lellis]{Camillo De Lellis}
\address{Institut f\"ur Mathematik, Universit\"at Z\"urich, CH-8057 Z\"urich}
\email{camillo.delellis@math.uzh.ch}

\author[De Rosa]{Luigi De Rosa}
\address{Institut f\"ur Mathematik, Universit\"at Z\"urich, CH-8057 Z\"urich}

\begin{abstract} 
We prove the ill-posedness of Leray solutions to the Cauchy problem for the ipodissipative Navier--Stokes equations,
when the dissipative term is a fractional Laplacian $(-\Delta)^\alpha$ with exponent $\alpha < \frac{1}{5}$. The proof follows the
``convex integration methods'' introduced by the second author and L\'aszl\'o Sz\'ekelyhidi Jr. for the incomprresible Euler equations.
The methods yield indeed some conclusions even for exponents in the range $[\frac{1}{5}, \frac{1}{2}[$. 
\end{abstract}

\maketitle



\section{Introduction}

In this paper we consider the ipodissipative Navier--Stokes equations on a periodic $3$-dimensional torus, namely the system
\begin{equation}\label{NS}
\left\{\begin{array}{l}
\partial_t v + \div (v\otimes v) + \nabla p+(-\Delta)^\alpha v =0\\  \hspace{8cm}\text{in } \mathbb{T}^3  \times [0,1] \\
\div v = 0
\end{array}\right.
\end{equation}
where $\alpha \in ]0,1[$ and $- (-\Delta)^\alpha$ is the fractional Laplacian operator, which in Fourier series has the symbol $- |k|^{2\alpha}$:
\[
- (- \Delta)^\alpha f (x) = - \sum_{k\in \mathbb Z^3} |k|^{2\alpha} \hat{f}_k e^{ik\cdot x}\, \qquad \forall f\in \mathcal{D}' (\T^3)\, .
\]
As for the classical Navier--Stokes equations, the celebrated method of Leray can be applied to the Cauchy problem for system \eqref{NS} in order to produce solutions which satisfy a suitable energy inequality. More precisely we have the following theorem.

\begin{theorem}\label{t:leray}
For any $\overline v \in L^2 (\T^3)$ with $\div \overline v=0$ and every $\alpha \in ]0,1[$ there is a weak solution $u\in L^\infty (\R^+, L^2 (\T^3)) \cap L^2 (\R^+, H^\alpha (\T^3))$ of \eqref{NS} such that $v (\cdot, 0) = \overline v$ and 
\begin{equation}\label{e:energy_ineq}
\frac{1}{2}\int_{\T^3} |v|^2 (x, t)\, dx + \int_0^t \int_{\T^3} |(-\Delta)^{\sfrac{\alpha}{2}} v|^2 (x, s)\, dx\, ds \leq \frac{1}{2}\int_{\T^3} |\overline v|^2 (x)\, dx \quad \forall t\geq 0\, .
\end{equation}
\end{theorem}

For the reader's convenience we will include a proof of Theorem \ref{t:leray} in the appendix.
As usual, the term {\em weak solution} of \eqref{NS} with initial data $\overline v$ is used for any solenoidal vector field $v$ such that
\[
\int_0^\infty \int_{\T^3} [(\partial_t - (-\Delta)^\alpha) \varphi \cdot v + D \varphi : v\otimes v] (x,s) \, dx\, ds
= - \int_{\T^3} \overline v (x)\cdot \varphi (x,0)\, dx
\]
for every smooth test vector field $\varphi \in C^\infty_c (\T^3\times \R, \R^3)$ with $\div \varphi = 0$. Note that $p$ can be recovered uniquely (as a distribution) if we impose that $\int p (x,t)\, dx = 0$. 

It is not difficult to show that any weak solution of \eqref{NS} in $L^\infty (\R^+, L^2 (\T^3)) \cap L^2 (\R^+, H^\alpha (\T^3))$ can be redefined on a set of measure zero so that the map $\R^+\ni t\mapsto v (\cdot, t)\in L^2 (\T^3)$ is weakly continuous. The spatial $L^2$ norm of the solution is thus well defined for every time $t$: \eqref{e:energy_ineq} must be interpreted in a pointwise-in-time sense using the corresponding well defined trace $v (\cdot, t)$. 
As it is the case in Leray's construction for the ``classical'' Navier--Stokes equations, the solution produced by the proof of Theorem \ref{t:leray} can be shown to satisfy an additional form of the energy estimate, namely:
\begin{align}
& \frac{1}{2} \int_{\T^3} |v|^2 (x,t)\, dx + \int_s^t\int_{\T^3} |(-\Delta)^{\sfrac{\alpha}{2}} v|^2 (x, \tau)\, dx\, d\tau \leq \frac{1}{2} \int_{\T^3} |v|^2 (x,s)\, dx\nonumber\\
&\qquad\qquad\qquad\qquad\qquad\qquad\qquad\qquad\qquad\qquad\qquad\qquad\mbox{for a.e. $s$ and $\forall t>s$} \, .\label{e:energy_ineq_2}
\end{align}
From now on, solutions of the Cauchy problem $v (\cdot, 0) = \overline v$ of \eqref{NS} defined on $\mathbb T^3 \times \R^+$ and satisfying \eqref{e:energy_ineq} and \eqref{e:energy_ineq_2} will be called Leray solutions.

\medskip

In this note we show that the ``convex integration'' methods introduced in \cite{DS-Inv} can be used to disprove the uniqueness of Leray's solutions if the exponent $\alpha$ is sufficiently small.

\begin{theorem}\label{t:nonuniq}
Let $\alpha < \frac{1}{5}$. Then there are initial data $\overline v\in L^2 (\T^3)$ with $\div \overline v=0$ for which there exist infinitely many Leray solutions $v$ of \eqref{NS} with $v (\cdot, 0) = \overline v$.
\end{theorem}

Indeed, the solutions $v$ constructed in our proof are somewhat stronger in a sufficiently small interval containing the origin. More precisely we prove the following

\begin{theorem}\label{t:locale}
Let $\alpha < \frac{1}{5}$. Then there are initial data $\overline v\in L^2 (\T^3)$ with $\div \overline v =0$ such that
\begin{itemize}
\item[(a)] $\overline v$ belongs to some H\"older space $C^\beta (\T^3)$ for $\alpha < \beta < \frac{1}{5}$;
\item[(b)] there is a positive time $T$  and infinitely many solutions $v \in C^\beta (\T^3 \times [0,T])$ of \eqref{NS} with $v (\cdot, 0) = \overline v$;
\item[(c)] such solutions satisfy the energy inequality \eqref{e:energy_ineq_2} for \em{all times} $0 \leq s\leq t \leq T$.
\end{itemize}
\end{theorem}

Each solution in Theorem \ref{t:locale} can be prolonged past the time $T$ using Theorem \ref{t:leray} (note that \eqref{NS} is invariant under time-shifts and so Theorem \ref{t:leray} is valid with any initial time $T$ substituting $0$): Theorem \ref{t:nonuniq} is thus an obvious corollary. Moreover, the solutions constructed in our proof can be arranged so to violate the energy {\em equality}, namely the inequality in \eqref{e:energy_ineq_2} can be shown to be strict for some times (see Remark \ref{r:stretta}).

The main point of the proof of Theorem \ref{t:locale} is that the methods introduced in \cite{DS-Inv} for the incompressible Euler equations and developed further in the literature (especially in the context of Onsager's conjecture, see \cite{DS-JEMS,Is2013,BDIS,Buck,BDS,IV,DaSz,Is2016,BDSV}) can be adapated to produce infinitely many local solutions satisfying (a), (b) and (c). More specifically, our proof is a simple modification of the one in \cite{BDIS}. As we will see, the type of iteration used in \cite{BDIS} works indeed when the exponent $\alpha$ is smaller than $\frac{1}{2}$, in particular it yields infinitely many weak solutions even in the range $\alpha \in [\frac{1}{5}, \frac{1}{2}[$. In the latter case, however, we are not able to show that such solutions satisfy the corresponding energy inequalities: therefore they are {\em not} Leray solutions. 

In the forthcoming paper \cite{DR} the second author will extend the validity of Theorem \ref{t:locale} to the range of H\"older exponents $]0, \frac{1}{3}[$, combining the ideas of this paper with those of \cite{BDSV} (the latter reference builds on the new techniques introduced in \cite{DaSz,Is2016}, which led Isett in \cite{Is2016} to finally prove the Onsager's conjecture).
However, since the arguments in \cite{DR} will be much longer and more complicated, we hope that the current note will help the interested readers in understanding the simple mechanisms behind the Theorems \ref{t:nonuniq} and \ref{t:locale}, once the convex integration methods are taken for granted.

\medskip

The key starting observation that the addition of a (sufficiently weak) ipodissipative term does not obtstruct the convex integration methods (introduced for the Euler equations) is indeed due to Buckmaster, Shkoller and Vicol in \cite{BSV}, although in a different context. The main difficulties here are:
\begin{itemize}
\item to ensure that the energy condition of Leray's weak solutions can be fullfilled;
\item to ensure that one can impose the same initial data to infinitely many solutions.
\end{itemize}
The first point requires a careful estimate of the H\"older norm of the solutions. The second point has been addressed in wide generality in the papers \cite{Da} and \cite{DaSz} for the Euler equations. Here we solve the issue with a very simple trick, avoiding pages of lengthy arguments.

\medskip

In the remarkable works \cite{JS-Inv,JS} the authors have conjectured (and given strong evidence) that even Leray solutions of the classical Navier--Stokes equations (namely with $\alpha =1$) are not unique. However, the mechanism suggested in \cite{JS} is entirely different from the one exploited here. 

\section{Local ill-posedness}

In this section we outline the main argument leading to Theorem \ref{t:locale}. In fact we will show a somewhat more general result, where the exponent $\alpha$ is taking values in the range $]0, \frac{1}{2}[$. More precisely we will show that

\begin{theorem}\label{t:main}
Assume $e: [0, 1]\to \R$ is a positive smooth function with $\frac{1}{2} \leq e(t) \leq 1$ and $\varepsilon > 0$ a positive number.
For any $\alpha \in ]0,\frac{1}{2}[$  there is a solution $(v,p)\in C^0( \T^3 \times [0, 1]; \R^3\times \R)$ of \eqref{NS}
such that
\begin{equation}\label{e:energy_id}
e(t) = \int_{\mathbb T^3} |v|^2 (x,t)\, dx\qquad\forall t\in [0, 1]\, 
\end{equation}
and
\begin{itemize}
\item[(i)] $v\in C^{\frac{1}{5}-\eps}$, $p\in  C^{\frac{2}{5}-2\eps }$ if $\alpha \leq \frac{1}{4}$,
\item[(ii)] $v\in C^{\frac{1-2\alpha}{3-2\alpha}-\eps}$ and $p\in  C^{2 \frac{1-2\alpha}{3-2\alpha}-2\eps }$ if $\frac{1}{4}< \alpha < \frac{1}{2} \,$.
\end{itemize}  
\end{theorem}

A crucial point is that the {\em{argument}} producing the pair $(v,p)$ of Theorem \ref{t:main}
gives two additional pieces of information, summarized in the following Proposition.

\begin{proposition}\label{p:tecnica}
Let $E_1, E_2>1$. Assume $\mathscr{E}$ is a family of smooth functions on $[0,1]$ with the property that 
\begin{itemize}
\item[(i)] $\frac{1}{2} \leq e (t) \leq 1$ for every $t$ and every $e\in \mathscr{E}$;
\item[(ii)] $e (0)$ is the same for every $e\in \mathscr{E}$;
\item[(iii)] $e' (0)$ is the same for every $e\in \mathscr{E}$;
\item[(iv)] $\sup_{e\in \mathscr{E}} \|e\|_{C^1} = E_1$;
\item[(v)] $\sup_{e\in \mathscr{E}} \|e\|_{C^2} = E_2$.
\end{itemize}
Then for each $e\in \mathscr{E}$ it is possible to produce a corresponding pair $(v_e, p_e)$ for which the following holds.
\begin{itemize} 
\item[(a)] $(v_e, p_e)$ solves \eqref{NS};
\item[(b)] each $v_e$ satisfies 
\begin{equation}\label{e:identita_E}
e (t) = \int_{\mathbb T^3} |v_e|^2 (x,t)\, dx\qquad\forall t\in [0, 1]\,
\end{equation}
\item[(c)] If $\alpha < \alpha + \varepsilon < \frac{1}{5}$ and $\varepsilon$ is suitably small (depending only upon $\alpha$), then we have
the explicit estimate
\begin{equation}\label{e:chiave}
\|v_e\|_{C^{\alpha+\varepsilon}} \leq C (\alpha, \varepsilon) \max \left\{E_1^{2\alpha+3\varepsilon},  E_2^{\frac{2\alpha +4\varepsilon}{3}}\right\}\, ;
\end{equation}
\item[(d)] The initial data $v_e (\cdot, 0)$ is the same for every $e\in \mathscr{E}$.  
\end{itemize}
\end{proposition}
 
Proposition \ref{p:tecnica} easily implies Theorem \ref{t:locale}.

\begin{proof}[Proof of Theorem \ref{t:locale}]
We fix $\alpha < \frac{1}{5}$ and choose $\alpha+\varepsilon \in ]\alpha, \frac{1}{5}[$ so that 
\eqref{e:chiave} holds. Elementary arguments produce for every $K>1$ an  infinite set $\mathscr{E}_K$ of smooth functions $e: [0,1]\to \mathbb R$ with the following properties:
\begin{itemize}
\item[(i)] $\frac{1}{2} \leq e(t) \leq 1$ $\forall t$;
\item[(ii)] $\|e\|_{C^1} \leq 2K+2$;
\item[(iii)] $e (0) =1$ and $e' (0) = -2K$;
\item[(iv)] $e' (t) \leq - 2K +2$ $\forall t \in [0, \frac{1}{4K}]$;
\item[(v)] $\|e\|_{C^2} \leq C K^2$, where $C$ is a geometric constant independent of $K$;
\item[(vi)] for any pair of distinct elements of $\mathscr{E}_K$ there is a sequence of times converging to
$0$ where they take different values.
\end{itemize}
We can now use Proposition \ref{p:tecnica} and for each energy profile $e\in \mathscr{E}_K$ we get a corresponding pair
$(v, p)$ of solutions of \eqref{NS} with \eqref{e:energy_id}. We claim that these solutions satisfy the energy inequality
\begin{equation}\label{e:energia_puntuale}
\frac{1}{2} \int_{\T^3} |v|^2 (x,t)\, dx + \int_s^t \int_{\T^3} |(-\Delta)^{\sfrac{\alpha}{2}} v|^2 (x, \tau)\, dx d\tau \leq \frac{1}{2} \int_{\T^3} |v|^2 (x,s)\, dx
\qquad \forall 0\leq s\leq \tau \leq \frac{1}{4K}\, ,
\end{equation}
provided $\varepsilon$ is chosen first sufficiently small and $K$ is then chosen large enough (depending on the two fixed exponents $\alpha$ and $\alpha + \varepsilon$). Recall that by Proposition \ref{p:tecnica} all such solutions have the same initial data $v (\cdot, 0) = \overline v$. Moreover, by (vi) they are all
distinct on $[0, \frac{1}{4K}]$. 

In order to show \eqref{e:energia_puntuale}, observe that by (iv) and \eqref{e:energy_id} we just need to show that
\begin{equation}\label{e:quasi}
\int_{\T^3} |(-\Delta)^{\sfrac{\alpha}{2}} v|^2 (x, \tau)\, dx \leq K -1 \qquad \forall \tau \in [0, \textstyle{\frac{1}{4K}}]\, .
\end{equation}
On the other hand by Corollary \ref{vlaplv} we have
\[
\int_{\T^3} |(-\Delta)^{\sfrac{\alpha}{2}} v|^2 (x, \tau)\, dx \leq C (\alpha, \varepsilon) \|v\|_{C^{\alpha+\varepsilon}}^2\, .
\]
By (i) we can use the estimate \eqref{e:chiave} and combine it with (ii) and (v) above to conclude
\[
\int_{\T^3} |(-\Delta)^{\sfrac{\alpha}{2}} v|^2 (x, \tau)\, dx \leq C (\alpha, \varepsilon) \max
\left\{K^{4\alpha + 6\varepsilon}, K^{\frac{8\alpha +16\varepsilon}{3}}\right\}\, .
\]
We next fix $\varepsilon$ so small that 
$\gamma := \max \{4\alpha + 6\varepsilon, \frac{8\alpha+16\varepsilon}{3}\} < 1 $.
Hence, we conclude 
\[
\int_{\T^3} |(-\Delta)^{\sfrac{\alpha}{2}} v|^2 (x, \tau)\, dx \leq C (\alpha) K^{\gamma}\, .
\]
Since $\alpha$ is fixed, choosing $K$ large enough we clearly achieve \eqref{e:quasi}.
\end{proof}

\begin{remark}\label{r:stretta}
Clearly, for $K$ large enough we can impose the inequality
\[
\int_{\T^3} |(-\Delta)^{\sfrac{\alpha}{2}} v|^2 (x, \tau)\, dx \leq K -2 \qquad \forall \tau \in [0, \textstyle{\frac{1}{4K}}]\, .
\]
in place of \eqref{e:quasi}, thus showing that the inequality in \eqref{e:energia_puntuale} can be made strict.
\end{remark}

It is worth to note that some conclusion can also be drawn in the range $\alpha \in [\frac{1}{5}, \frac{1}{2}[$. More precisely, the argument given above can be easily modified to prove the following

\begin{corollary}
Let $\alpha \in [\frac{1}{5}, \frac{1}{2}[$. Then there are initial data $\overline v\in C (\mathbb T^3)$ with $\div \overline v=0$ for which there exist infinitely many weak solution $v\in L^\infty ([0, \infty[, L^2 (\T^3))$ of \eqref{NS} with $v (\cdot, 0)=\overline v$.  
\end{corollary}

\section{Main iteration scheme}

The proof of Theorem \ref{t:main} is achieved through an iteration scheme. At each step $q\in \N$ we construct a triple $(v_q, p_q, \mathring{R}_q)$ solving the Fractional Navier-Stokes-Reynolds system:
\begin{equation}\label{fracNSR}
\left\{\begin{array}{l}
\partial_t v_q + \div (v_q\otimes v_q) + \nabla p_q+ (-\Delta)^\alpha v_q =\div\mathring{R}_q\\ \\
\div v_q = 0\, .
\end{array}\right.
\end{equation} 
The $3\times 3$ symmetric traceless tensor $\mathring{R}_q$ is related to the so-called Reynolds stress, a quantity which arises
naturally when considering highly oscillatory solutions of the Euler equations. The scheme will be set up so that
$\mathring{R}_q$ converges uniformly to $0$, whereas the pair $(v_q, p_q)$ converges uniformly to the pair $(v,p)$ of Theorem \ref{t:main}. 

The {\it size} of the perturbation 
$$
w_q:=v_{q}-v_{q-1}
$$
will be measured by two parameters: $\delta_q^{\sfrac12}$ is the {\it amplitude} and $\lambda_q$ the {\it frequency}. More precisely, denoting the (spatial) H\"older norms by $\|\cdot\|_k$ ,
\begin{align}
\|w_{q}\|_0 &\leq M \delta_{q}^{\sfrac{1}{2}}\,,\label{e:v_C0_iter}\\
\|w_{q}\|_1 &\leq M \delta_{q}^{\sfrac{1}{2}} \lambda_{q}\,,\label{e:v_C1_iter}
\end{align}
and similarly,
\begin{align}
\|p_{q}-p_{q-1}\|_0 &\leq M^2 \delta_{q}\,, \label{e:p_C0_iter}\\
\|p_{q}-p_{q-1}\|_1 &\leq M^2 \delta_{q} \lambda_{q}\,,\label{e:p_C1_iter}
\end{align}
where $M$ is a constant depending only on the function $e=e(t)$ (cf.~Section \ref{s:eta-M}), more specifically only upon $\max e$ and $\min e$, which by our assumptions are anyway under control. Thus in the rest of the note $M$ will be treated as a fixed geometric constant.

In constructing the iteration, the new perturbation $w_q$ will be chosen so as to balance the 
previous Reynolds error $\mathring{R}_{q-1}$ in the sense that (cf. equation \eqref{fracNSR}) 
we have $\|w_q\otimes w_q\|_0\sim \|\mathring{R}_{q-1}\|_0$. To make this possible, we then claim inductively the estimates 
\begin{align}
\|\mathring{R}_q\|_0 &\leq \eta \delta_{q+1}\,,\label{e:R_C0_iter}\\
\|\mathring{R}_q\|_1&\leq M \delta_{q+1}\lambda_q\,, \label{e:R_C1_iter}\\
\|\partial_t \mathring{R}_q + v_q \cdot \nabla \mathring{R}_q\|_0 & \leq \delta_{q+1} \delta_q^{\sfrac{1}{2}} \lambda_q\, , \label{e:R_Dt_iter}
\end{align}
where $\eta$ will be a small constant depending indeed only upon $\max e$ and $\min e$ (cf. again Section \ref{s:eta-M}). Thus, similarly to $M$, $\eta$ can be treated as a fixed absolute constant. 

Along the iteration we will have 
\begin{equation}\label{e:delta-lambda}
\delta_q = a^{-b^q}\textrm{ and }\quad \lambda_q \in  \mathbb N \cap [a^{cb^{q+1}}, 2 a^{cb^{q+1}}]\, ,
\end{equation}
where the constants $b$ and $c$ are fixed and satisfy $b>1$ and $c>\frac{5}{2}$, whereas $a$ will be chosen (depending on $b,c$, $\alpha$ and $e$) much larger than $1$. On the one hand \eqref{e:v_C0_iter}, \eqref{e:p_C0_iter} and \eqref{e:R_C0_iter} will imply the convergence of the sequence $v_q$ to a continuous weak solution of \eqref{fracNSR}. On the other hand the precise dependence of $\lambda_q$ on $\delta_q$ will determine the critical H\"older regularity. Finally, the equation \eqref{e:energy_id} will be ensured by 
\begin{equation}\label{e:energy_iter}
\left| e(t) (1-\delta_{q+1}) - \int |v_q|^2 (x,t)\, dx \right| \leq \frac{1}{4} \delta_{q+1} e(t)\,.
\end{equation}

\subsection{The starting triple} In this section we specify the starting triple $(v_0, p_0, \mathring{R}_0)$. 

\begin{lemma}\label{l:partenza}
Fix $M$ and $\eta$ positive constants and
let $\alpha\in ]0, \frac{1}{5}[$. If $a,b$ and $c$ satisfy the following conditions
\begin{equation}\label{e:condizioni_su_abc}
c>\frac{5}{2}, b>1, a^{(c-1)b-\frac{1}{2}} \geq C_0 \|e\|_{C^1}\quad \mbox{and}\quad
a^{(2c-1)b -1} \geq C_0 \|e\|_{C^2}
\end{equation}
(where $C_0$ is a suitable geometric constant, depending only upon $M$ and $\eta$),
then there is a triple $(v_0, p_0, \mathring{R}_0)$ satisfying \eqref{fracNSR}, \eqref{e:R_C0_iter}, \eqref{e:R_C1_iter},
\eqref{e:R_Dt_iter} and
\begin{align}
\|v_0\|_0 &\leq M\label{e:v_start},\\
\|v_0\|_1 &\leq \min \left\{ C_0 \max \left\{a^{\frac{b}{1-2\alpha}}, a^b \|e\|_{C^1}, \|e\|_{C^2}^{\frac{cb - 1/2}{(2c-1)b -1}}\right\},  M \delta_0^{\sfrac{1}{2}} \lambda_0\right\}, \label{e:v_1_start}\\
\|p_0\|_0 &\leq M^2, \label{e:p_start}\\
\|p_0\|_1 &\leq M^2 \delta_0 \lambda_0^2\label{e:p_1_start}\, .
\end{align}
For $\alpha \in [\frac{1}{5}, 1[$ there is a starting triple satisfying all the above estimates with
\begin{equation}\label{e:v_1_start_alternativa}
\|v_0\|_1 \leq M \delta_0^{\sfrac{1}{2}} \lambda_0\, 
\end{equation}
in place of \eqref{e:v_1_start},
provided $c>\max \{\frac{5}{2}, \frac{3-2\alpha}{2(1-2\alpha)}\}$, $b>1$ and $a$ is chosen large enough depending only upon $\|e\|_{C^2}$, $\alpha$, $b$ and $c$
\end{lemma}

\begin{proof} In the rest of the proof we will use the notation $C_0$ for constants which are independent of any parameter (but might depend on the constants $M$ and $\eta$)
. We only check the case $\alpha < \frac{1}{5}$, since indeed the other case is much simpler. 

Observe that $\delta_1 = a^{-b} <1$. We define $p_0=0$, 
\[
v_0(x,t)=\frac{1}{(2\pi)^{\sfrac{3}{2}}}(e(t)(1-\delta_1))^{\sfrac{1}{2}}(\cos \bar \lambda x_3,\sin \bar \lambda x_3,0)
\] 
and $\mathring{R}_0=\mathring{R}_{0,1}+\mathring{R}_{0,2}$, where
\[
\mathring{R}_{0,1}(x,t)=\frac{1}{(2\pi)^{\sfrac{3}{2}}} \bar \lambda ^{-1} \frac{d}{dt}(e(t)(1-\delta_1))^{\sfrac{1}{2}}\begin{bmatrix} 0& 0& \sin \bar \lambda x_3 \\ 0&
0 & -\cos \bar \lambda x_3  \\ \sin \bar \lambda x_3  &-\cos \bar \lambda x_3  &
0\end{bmatrix}
\]
\[
\mathring{R}_{0,2}(x,t)= \frac{1}{(2\pi)^{\sfrac{3}{2}}}\bar \lambda ^{-1+2\alpha} (e(t)(1-\delta_1))^{\sfrac{1}{2}}\begin{bmatrix} 0& 0& \sin \bar \lambda x_3 \\ 0&
0 & -\cos \bar \lambda x_3  \\ \sin \bar \lambda x_3  &-\cos \bar \lambda x_3  &
0\end{bmatrix}
\]
and $\bar \lambda$ is an integer whose choice will be specified later.

\medskip

Note that \eqref{e:p_start} and \eqref{e:p_1_start} are trivial, whereas \eqref{fracNSR} can be easily checked. We now come to the other estimates. 

\medskip

{\bf Proof of \eqref{e:R_C0_iter}.} We require separately $\|\mathring{R}_{0,1}\|_0 \leq \frac{\eta}{2} \delta_1$
and $\|\mathring{R}_{0,2}\|_0 \leq \frac{\eta}{2} \delta_1$. These two estimates are certainly satisfied
provided
\begin{align}
\bar\lambda &\geq C_0 \|e\|_{C^1}\delta_1^{-1} \label{e:una}\\
\bar\lambda &\geq C_0 \delta_1^{-\frac{1}{1-2\alpha}}\label{e:due}\, .
\end{align}

\medskip

{\bf Proof of \eqref{e:R_C1_iter}.} We require separately $\|\mathring{R}_{0,1}\|_1 \leq \frac{M}{2} \delta_1 \lambda_0$ and 
$\|\mathring{R}_{0,2}\|_1 \leq \frac{M}{2} \delta_1 \lambda_0$. These are certainly satisfied if
\begin{align}
\delta_1 \lambda_0 &\geq C_0 \|e\|_{C^1}\label{e:tre}\\
\delta_1 \lambda_0 &\geq C_0 \bar\lambda^{2\alpha}\label{e:quattro}\, .
\end{align}

\medskip

{\bf Proof of \eqref{e:R_Dt_iter}.} Observe that $(v_0\cdot \nabla)\mathring{R}_0 =0$. Thus it suffices to require the two estimates $\|\partial_t \mathring{R}_{0,1}\|_0 \leq \frac{1}{2} \delta_1 \delta_0^{\sfrac{1}{2}} \lambda_0$ and
$\|\partial_t \mathring{R}_{0,2}\|_0 \leq \frac{1}{2} \delta_1 \delta_0^{\sfrac{1}{2}} \lambda_0$. These are certainly achieved if we impose
\begin{align}
\bar \lambda &\geq C_0 \|e\|_{C^2} (\delta_1 \delta_0^{\sfrac{1}{2}} \lambda_0)^{-1}\label{e:cinque}\\
\bar \lambda&\geq C_0 \left(\|e\|_{C^1} (\delta_1 \delta_0^{\sfrac{1}{2}} \lambda_0)^{-1}\right)^{\frac{1}{1-2\alpha}}
\label{e:sei}
\end{align}

\medskip

{\bf Conclusion.} \eqref{e:v_start} is obvious since $\|v_0\|_0 \leq 1$ and $M\geq 1$. The inequality \eqref{e:v_1_start} will be split into two conditions. One is
\begin{equation}\label{e:sette}
 \bar \lambda \leq \delta_0^{\sfrac{1}{2}} \lambda_0\, ,
\end{equation}
whereas the other one is
\begin{equation}\label{e:otto}
\bar\lambda \leq C_0 \max \left\{a^{\frac{b}{1-2\alpha}}, a^b \|e\|_{C^1}, \|e\|_{C^2}^{\frac{cb - 1/2}{(2c-1)b -1}}\right\}\, .
\end{equation}
The conditions \eqref{e:una}, \eqref{e:cinque} and \eqref{e:sei} determine the choice of $\bar\lambda$, which we fix to be just the maximum of all the right hand sides of the respective conditions. In fact, given the definition of $\delta_q$'s and $\lambda_q$'s, we just have
\[
\bar\lambda = C_0 \max \left\{ a^{\frac{b}{1-2\alpha}}, a^b \|e\|_{C^1}, \|e\|_{C^1}^{\frac{1}{1-2\alpha}}
a^{- ((c-1)b -1/2)/(1-2\alpha)}, \|e\|_{C^2} a^{- (c-1) b + 1/2}\right\}\, .
\]
We next need to check that, given the inequalities required on $a$, the conditions \eqref{e:tre}, \eqref{e:quattro}, \eqref{e:sette} and \eqref{e:otto} are satisfied. First of all, notice that \eqref{e:tre} is satisfied because it is equivalent to
\[
a^{(c-1)b} \geq C_0 \|e\|_{C^1}\, ,
\]
which is satisfied by \eqref{e:condizioni_su_abc}. The latter inequality shows easily that 
\[
a^b \|e\|_{C^1} \geq \|e\|_{C^1}^{\frac{1}{1-2\alpha}}
a^{- ((c-1)b -1/2)/(1-2\alpha)}\, ,
\]
so that we can simplify the definition of $\bar \lambda$ to
\begin{equation}\label{e:barlambda}
\bar\lambda = C_0 \max \left\{ a^{\frac{b}{1-2\alpha}}, a^b \|e\|_{C^1}, \|e\|_{C^2} a^{- (c-1) b + 1/2}\right\}
\end{equation}
Next, we check that \eqref{e:quattro} holds, which amount to check separately that
\begin{align}
a^{(c-1) b} &\geq a^{\frac{2\alpha b}{1-2\alpha}}\\
a^{(c-1) b} &\geq a^{2\alpha b}\|e\|_{C^1}^{2\alpha}\\
a^{(c-1) b} &\geq a^{-2\alpha (c-1)b +\alpha} \|e\|_{C^2}^{2\alpha}\, .
\end{align}
Now, the first inequality is obvious because 
\[
(c-1) > \frac{3}{2} = \frac{\frac{2}{5}}{1-\frac{2}{5}} > \frac{2\alpha}{1-2\alpha}\, .
\]
The second inequality is equivalent to
\[
a^{(c-1-2\alpha) b} \geq \|e\|_{C^1}^{2\alpha}\, ,
\]
which is implied by $\|e\|_{C^1} \leq a^{(c-1)b}$ (to pass from one to the other we again use $(c-1) \geq \frac{2\alpha}{1-2\alpha}$). 
The third inequality is implied by 
\[
a^{(c-1)(1+2\alpha) b - \alpha} \geq \|e\|_{C^2}^{2\alpha}\, ,
\]
which is indeed guaranteed by \eqref{e:condizioni_su_abc}, because $(c-1)(1+2\alpha)b /2\alpha - \frac{1}{2}
\geq (2c-1)b -1$. 

We next check \eqref{e:sette}. The latter is equivalent to
\begin{align}
a^{cb-\sfrac{1}{2}} \geq C_0 a^{\frac{b}{1-2\alpha}}\\
a^{cb-\sfrac{1}{2}} \geq C_0 a^b\|e\|_{C^1}\\
a^{cb-\sfrac{1}{2}} \geq C_0 a^{-(c-1)b+\sfrac{1}{2}} \|e\|_{C^2}\, .
\end{align}
The first one is trivially implied by $\alpha \leq \frac{1}{5}$ and $c\geq \frac{5}{2}$. The second is equivalent to
$a^{(c-1)b-\frac{1}{2}} \geq C_0 \|e\|_{C^1}$, which is indeed in \eqref{e:condizioni_su_abc}. The last one is equivalent to
\[
a^{(2c-1)b -1}\geq C_0 \|e\|_{C^2}\, .
\]
Inserting the latter inequality into \eqref{e:barlambda} we achieve \eqref{e:otto}, which completes the proof.
\end{proof}

\subsection{The main iteration and the proof of Theorem \ref{t:main}}

Given  the triple $(v_0,p_0,\mathring{R}_0)$ provided by Lemma \ref{l:partenza} we will construct inductively new triples $(v_q,p_q,\mathring{R}_q)$, assuming the estimates \eqref{e:v_C0_iter}-\eqref{e:R_Dt_iter}. Such iterative scheme will then lead to the following Proposition.

\begin{proposition}\label{p:iterate}
There are positive constants $M\geq 1$, $\eta>0$ and $C_0$ such that the following holds.
\begin{itemize}
\item[$\bullet$] Assume  $\alpha<\frac{1}{5}$ and $a,b$ and $c$ satisfy 
\begin{align}
& c>\frac{5}{2}, b>1\quad \mbox{and}\quad  a\geq \max \left\{ a_0 (b,c), C_0 \|e\|_{C^1},
C_0 \|e\|_{C^2}^{\frac{1}{(2c-1)b-1}}\right\}\, ,\label{e:condizioni_su_abc_2}
\end{align}
where $a_0$ depends only upon $b$ and $c$. Then there is a sequence
$(v_q, p_q, \mathring{R}_q)$ starting with the $(v_0, p_0, \mathring{R}_0)$ of Lemma \ref{l:partenza},
solving \eqref{NS} and satisfying the estimates \eqref{e:v_C0_iter}-\eqref{e:R_Dt_iter}, where 
$\delta_q$ and $\lambda_q$ are as in \eqref{e:delta-lambda}. 
\item[$\bullet$]If $\alpha \in [\frac{1}{5},\frac{1}{2}[$. then the same as above holds if
$c>\max \{\frac{5}{2}, \frac{3-2\alpha}{2(1-2\alpha)}\}$, $b>1$ and $a$ is chosen large enough depending only upon $\|e\|_{C^2}$, $\alpha$, $b$ and $c$.
\end{itemize} 
In addition we claim the estimates
\begin{align}
\|\partial_t (v_q-v_{q-1})\|_0 \leq C \delta_q^{\sfrac{1}{2}}\lambda_q \qquad\mbox{and}\qquad
\|\partial_t (p_q-p_{q-1})\|_0\leq C \delta_q \lambda_q\label{e:t_derivatives}\, .
\end{align}
\end{proposition}

Theorem \ref{t:main} is a very easy consequence of the above Proposition and we give the argument immediately.
Proposition \ref{p:tecnica} is somewhat more involved, since in fact it needs the details of the arguments of
Proposition \ref{p:iterate}. For this reason we give the corresponding argument only at the very end of the
paper

\begin{proof}[Proof of Theorem \ref{t:main}]
 Let $(v_q, p_q,\mathring{R}_q)$ be a sequence as in Proposition \ref{p:iterate}.
It follows then easily that $\{(v_q, p_q)\}$ converge uniformly to a pair of continuous functions $(v,p)$ such that
\eqref{e:energy_id} holds. 
We introduce the notation $\|\cdot\|_{C^\vartheta}$ for H\"older norms in space and time.
From \eqref{e:v_C0_iter}-\eqref{e:p_C1_iter}, \eqref{e:t_derivatives} and interpolation we conclude
\begin{align}
\|v_{q+1}-v_q\|_{C^\vartheta} & \leq M \delta_{q+1}^{\sfrac{1}{2}} \lambda_{q+1}^\vartheta 
\leq C a^{b^{q+1} (2cb \vartheta -1)/2}\\
\|p_{q+1}-p_q\|_{C^{2\vartheta}} & \leq M^2 \delta_{q+1} \lambda_{q+1}^{2\vartheta} \leq 
C a^{b^{q+1} (2cb \vartheta -1)}\, .
\end{align}
Thus, for every $\vartheta< \frac{1}{2bc}$, $v_q$ converges in $C^{\vartheta}$ and $p_q$ in $C^{2\vartheta}$.
Now if $\alpha < \frac{1}{4}$, by Proposition \ref{p:iterate}, we can choose any constants $c>\frac{5}{2}$ and $b>1$, so we have convergence in $C^\vartheta$ for every $\vartheta <\frac{1}{5}$.
Otherwise, if $\frac{1}{4}\leq \alpha<\frac{1}{2}$ we can choose any $c > \frac{3-2\alpha}{2(1-2\alpha)}$ and $ b>1$, getting convergence for any $ \vartheta < \frac{1-2\alpha}{3-2\alpha}$.
\end{proof}

The rest of the paper is devoted to prove the Proposition \ref{p:iterate} and hence Proposition \ref{p:tecnica}. The concluding arguments will be given in the final section. 

\section{The main iteration}

In this section we specify the inductive procedure which builds $(v_{q+1}, p_{q+1}, \mathring{R}_{q+1})$ from $(v_q, p_q, \mathring{R}_q)$. Many steps follow literally the same construction in \cite{BDIS} and we repeat them for the reader's convenience. 

Note that the choice of the sequences $\{\delta_q\}_{q\in\N}$ and $\{\lambda_q\}_{q\in\N}$ specified in Proposition \ref{p:iterate} implies that, for $a> a_0 (b,c)$, we have:
\begin{equation}\label{e:summabilities}
\sum_{j\leq q} \delta_j \lambda_j \leq 2 \delta_q \lambda_q\, ,\quad
1 \leq \sum_{j\leq q} \delta_j^{\sfrac{1}{2}} \lambda_j \leq 2\delta_q^{\sfrac{1}{2}} \lambda_q\, ,\quad
 \sum_j \delta_j \leq \sum_j \delta_j^{\sfrac{1}{2}} \leq 2\, .
\end{equation}
Our inductive hypothesis together with Lemma \ref{l:partenza} imply then the following set of estimates:
\begin{align}
&\|v\|_0 \leq 2 M,\qquad \|v_q\|_1 \leq 2M \delta_{q}^{\sfrac{1}{2}} \lambda_{q}\,,\label{e:v_old}\\
&\|\mathring{R}_q\|_0 \leq \eta\delta_{q+1}, \,\quad \|\mathring{R}_q\|_1 \leq M\delta_{q+1}\lambda_q\,, \label{e:R_old}\\
&\|p_q\|_0 \leq 2 M^2,  \qquad \|p_q\|_1 \leq 2M^2 \delta_q \lambda_q\,,\label{e:p_old}
\end{align}
and
\begin{equation}
\|(\partial_t+v_q\cdot\nabla)\mathring{R}_q\|_0\leq M\delta_{q+1}\delta_q^{\sfrac12}\lambda_q\,.\label{e:DtR_old}
\end{equation}

\subsection{$v_{q+1}-v_q$ as a sum of modulated Beltrami flows}
We next recall the following two important facts, whose proof can be found in \cite{DS-Inv,BDIS}.

\begin{proposition}[Beltrami flows]\label{p:Beltrami}
Let $\bar\lambda\geq 1$ and let $A_k\in\R^3$ be such that 
$$
A_k\cdot k=0,\,|A_k|=\tfrac{1}{\sqrt{2}},\,A_{-k}=A_k
$$
for $k\in\Z^3$ with $|k|=\bar\lambda$.
Furthermore, let 
$$
B_k=A_k+i\frac{k}{|k|}\times A_k\in\C^3.
$$
For any choice of $a_k\in\C$ with $\overline{a_k} = a_{-k}$ the vector field
\begin{equation}\label{e:Beltrami}
W(\xi)=\sum_{|k|=\bar\lambda}a_kB_ke^{ik\cdot \xi}
\end{equation}
is real-valued, divergence-free and satisfies
\begin{equation}\label{e:Bequation}
\div (W\otimes W)=\nabla\frac{|W|^2}{2}.
\end{equation}
Furthermore
\begin{equation}\label{e:av_of_Bel}
\langle W\otimes W\rangle:= \fint_{\T^3} W\otimes W\,d\xi = \frac{1}{2} \sum_{|k|=\bar\lambda} |a_k|^2 \left( \Id - \frac{k}{|k|}\otimes\frac{k}{|k|}\right)\, .  
\end{equation}
\end{proposition}

\begin{lemma}[Geometric Lemma]\label{l:split}
For every $N\in\N$ we can choose $r_0>0$ and $\bar{\lambda} > 1$ with the following property.
There exist pairwise disjoint subsets 
$$
\Lambda_j\subset\{k\in \Z^3:\,|k|=\bar{\lambda}\} \qquad j\in \{1, \ldots, N\}
$$
and smooth positive functions 
\[
\gamma^{(j)}_k\in C^{\infty}\left(B_{r_0} (\Id)\right) \qquad j\in \{1,\dots, N\}, k\in\Lambda_j
\]
such that
\begin{itemize}
\item[(a)] $k\in \Lambda_j$ implies $-k\in \Lambda_j$ and $\gamma^{(j)}_k = \gamma^{(j)}_{-k}$;
\item[(b)] for each $R\in B_{r_0} (\Id)$ we have the identity
\begin{equation}\label{e:split}
R = \frac{1}{2} \sum_{k\in\Lambda_j} \left(\gamma^{(j)}_k(R)\right)^2 \left(\Id - \frac{k}{|k|}\otimes \frac{k}{|k|}\right) 
\qquad \forall R\in B_{r_0}(\Id)\, .
\end{equation}
\end{itemize}
\end{lemma}

The new velocity $v_{q+1}$ will be defined as a sum 
$$
v_{q+1}:= v_q + w_o + w_c,
$$
where $w_o$ is the principal perturbation and $w_c$ is a corrector.
 The ``principal part'' of the perturbation $w$ will be a sum of modulated Beltrami flows
 \[
w_o (t,x) := \sum_{|k|=\lambda_0} a_{k} (t,x) \phi_{k} (t,x) B_ke^{i\lambda_{q+1}k\cdot x}\, ,
\]
where $B_ke^{i\lambda_{q+1}k\cdot x}$ is a single Beltrami mode at frequency $\lambda_{q+1}$, with phase shift $\phi_{k}=\phi_{k}(t,x)$ (i.e. $|\phi_{k}|=1$) and amplitude $a_{k}=a_{k}(t,x)$. In the following subsections we will define $a_k$ and $\phi_k$.

\subsection{Space regularization of $v$ and $R$}\label{s:mollifier}

We fix a symmetric non-negative convolution kernel $\psi\in C^\infty_c (\R^3)$ and a small parameter $\ell$ (whose choice will be specified later). Define
$v_\ell:=v_q*\psi_\ell$ and $\mathring{R}_\ell:=\mathring{R}_q*\psi_\ell$,
where the convolution is in the $x$ variable only. 
Standard estimates on regularizations by convolution lead to the following:
\begin{align}
&\|v_q-v_{\ell}\|_0\leq C_0 M\,\delta_q^{\sfrac12}\lambda_q\ell,\label{e:v-v_ell}\\
&\|\mathring{R}-\mathring{R}_{\ell}\|_0\leq C_0 M\,\delta_{q+1}\lambda_q\ell\label{e:R-R_ell}
\end{align}
(where $C_0$ is a geometric constant)
and, for any $N\geq 1$, 
\begin{align}
&\|v_\ell\|_N\leq C M\,\delta_q^{\sfrac12}\lambda_q\ell^{1-N},\label{e:v_ell}\\
&\|\mathring{R}_\ell\|_N\leq C M\,\delta_{q+1}\lambda_q\ell^{1-N},\label{e:R_ell}
\end{align}
where $C$ is a constant which depends only upon $N$. 

\subsection{Time discretization and transport for the Reynolds stress}\label{s:time-discret}

Next, we fix a smooth cut-off function $\chi\in C^\infty_c ((-\frac{3}{4}, \frac{3}{4}))$ such that 
$$
\sum_{l\in \Z} \chi^2 (t-l) = 1,
$$ 
and a large parameter $\mu\in \N\setminus \{0\}$, whose choice will be specified later.

For any $l\in [0, \mu]$ we define
\[
\rho_l:= \frac{1}{3 (2\pi)^3} \left(e (l\mu^{-1}) \left(1-\delta_{q+2}\right) - \int_{\T^3} |v_q|^2 (x,l\mu^{-1})\, dx\right).
\]
Note that \eqref{e:energy_iter} implies
$$
\frac{1}{3(2\pi)^3}e (l\mu^{-1})(\tfrac{3}{4}\delta_{q+1}-\delta_{q+2})\leq \rho_l\leq \frac{1}{3(2\pi)^3}e (l\mu^{-1})(\tfrac{5}{4}\delta_{q+1}-\delta_{q+2}).
$$
Recalling that $b$ and $c$ are fixed, the condition $a\geq a_0 (b,c)$ implies that we might assume 
$\delta_{q+2}\leq \frac{1}{2}\delta_{q+1}$.
Thus we obtain
\begin{equation}\label{e:est_rhol}
\frac{\delta_{q+1}}{2C_0} \leq C_0^{-1}(\min e)\delta_{q+1}\leq \rho_l\leq C_0(\max e)\delta_{q+1} \leq 2 C_0 \delta_{q+1}\, ,
\end{equation}
where $C_0$ is (again) an absolute constant.

Finally, define $R_{\ell,l}$ to be the unique solution to the transport equation
\begin{equation}\label{e:R-transport}
\left\{\begin{array}{l}
\partial_t \mathring{R}_{\ell,l} +v_\ell\cdot\nabla \mathring{R}_{\ell,l} = 0 \\
\mathring{R}_{\ell,l}(x,\frac l{\mu})=\mathring{R}_{\ell}(x,\frac{l}{\mu})\, .
\end{array}\right.
\end{equation}
and set
\begin{equation}\label{e:mathringRnew}
R_{\ell,l}(x,t):=\rho_l\Id-\mathring{R}_{\ell,l}(x,t).
\end{equation}

\subsection{The maps $v_{q+1}, w, w_o$ and $w_c$}\label{ss:def_w}

We next consider $v_\ell$ as a $2\pi$-periodic function on $\R^3\times [0,1]$ and, for every $l\in [0, \mu]$, we let $\Phi_l: \R^3\times [0,1]\to \R^3$ be the solution of 
\begin{equation}\label{e:Phi-transport}
\left\{\begin{array}{l}
\partial_t \Phi_l + v_{\ell}\cdot  \nabla \Phi_l =0\\ \\
\Phi_l (x,l \mu^{-1})=x\, .
\end{array}\right.
\end{equation}
Observe that $\Phi_l (\cdot, t)$ is the inverse of the flow of the periodic vector-field $v_\ell$, starting at time $t=l\mu^{-1}$ as the identity. Thus, if $y\in (2\pi \Z)^3$, then $\Phi_l (x, t) - \Phi_l (x+y, t) \in (2\pi \Z)^3$:
$\Phi_l (\cdot, t)$ can hence be thought as a diffeomorphism of $\T^3$ onto itself and, for every $k\in \Z^3$, the map $\T^3 \times [0,1] \ni (x,t) \to e^{i \lambda_{q+1} k \cdot \Phi_l (x,t)}$ is well-defined.

We next apply Lemma \ref{l:split} with $N=2$, denoting by $\Lambda^e$ and $\Lambda^o$ the corresponding families of frequencies in $\Z^3$, and set $\Lambda := \Lambda^o$ + $\Lambda^e$. For each $k\in \Lambda$ and each $l\in \Z\cap[0,\mu]$ we then set
\begin{align}
\chi_l(t)&:=\chi\Bigl(\mu(t-\frac{l}{\mu})\Bigr),\label{e:chi_l}\\
a_{kl}(x,t)&:=\sqrt{\rho_l}\gamma_k \left(\frac{R_{\ell,l}(x,t)}{\rho_l}\right),\label{e:a_kl}\\
w_{kl}(x,t)& := a_{kl}(x,t)\,B_ke^{i\lambda_{q+1}k\cdot \Phi_l(x,t)}.\label{e:w_kl}
\end{align}
The ``principal part'' of the perturbation $w$ consists of the map
\begin{align}\label{e:def_wo}
w_o (x,t) := \sum_{\textrm{$l$ odd}, k\in \Lambda^o} \chi_l(t)w_{kl} (x,t) +
\sum_{\textrm{$l$ even}, k\in \Lambda^e} \chi_l(t)w_{kl} (x,t)\, .
\end{align}
From now on, in order to make our notation simpler, we agree that the pairs of indices 
$(k,l)\in \Lambda\times [0, \mu]$ which enter in our summations satisfy always 
the following condition: $k\in \Lambda^e$ when $l$ is even and $k\in \Lambda^o$ when $l$ is odd.

It will be useful to introduce the ``phase" 
\begin{equation}\label{e:phi_kl}
\phi_{kl}(x,t)=e^{i\lambda_{q+1}k\cdot[\Phi_l(x,t)-x]},
\end{equation}
with which we obviously have
\[
\phi_{kl}\cdot e^{i\lambda_{q+1}k\cdot x}=e^{i\lambda_{q+1}k\cdot\Phi_l}.
\]
Since $R_{\ell,l}$ and $\Phi_l$ are defined as solutions of the transport equations \eqref{e:R-transport} and \eqref{e:Phi-transport}, we have
\begin{align}\label{e:Dt_varia=0}
(\partial_t+v_\ell\cdot\nabla)a_{kl}=0\qquad\textrm{ and }\qquad(\partial_t+v_\ell\cdot\nabla)e^{i\lambda_{q+1}k\cdot \Phi_l(x,t)}=0,
\end{align}
hence also 
\begin{equation}\label{e:Dt_w_kl=0}
(\partial_t+v_\ell\cdot\nabla)w_{kl}=0.
\end{equation}
The corrector $w_c$ is then defined in such a way that $w:= w_o+w_c$ is divergence free:
\begin{align}
w_c&:= \sum_{kl} \frac{\chi_l}{\lambda_{q+1}}\curl\left(ia_{kl}\phi_{kl}\frac{k\times B_k}{|k|^2}\right) e^{i\lambda_{q+1}k \cdot x}\nonumber\\
&=\sum_{kl}\chi_l\Bigl(\frac{i}{\lambda_{q+1}}\nabla a_{kl}-a_{kl}(D\Phi_{l}-\Id)k\Bigr)\times\frac{k\times B_k}{|k|^2}e^{i\lambda_{q+1}k\cdot\Phi_l}\label{e:corrector}
\end{align}

\begin{remark} To see that $w=w_o+w_c$ is divergence-free, just note that, since $k\cdot B_k=0$, we have
$k\times (k\times B_k)=-|k|^2B_k$ and hence $w$ can be written as
\begin{equation}\label{e:w_compact_form}
w = \frac{1}{\lambda_{q+1}} \sum_{(k,l)} \chi_l \,\curl \left( i a_{kl}\,\phi_{kl}\,\frac{k\times B_k}{|k|^2} e^{i\lambda_{q+1}k \cdot x}\right)\, .
\end{equation}
\end{remark}

For future reference it is useful to introduce the notation
\begin{equation}\label{e:L_kl}
L_{kl}:=a_{kl}B_k+\Bigl(\frac{i}{\lambda_{q+1}}\nabla a_{kl}-a_{kl}(D\Phi_{l}-\Id)k\Bigr)\times\frac{k\times B_k}{|k|^2},
\end{equation}
so that the perturbation $w$ can be written as
\begin{equation}\label{e:w_Lform}
w=\sum_{kl}\chi_l\,L_{kl}\,e^{i\lambda_{q+1}k\cdot\Phi_l}\,.
\end{equation}
Moreover, we will frequently deal with the transport derivative with respect to the regularized flow $v_\ell$ of various expressions, and will henceforth use the notation
\begin{equation}\label{e:notation}
D_t:=\partial_t+v_{\ell}\cdot\nabla.
\end{equation}

\subsection{Determination of the constants $\eta$ and $M$}\label{s:eta-M}
In order to determine $\eta$, first of all recall from Lemma \ref{l:split} that the functions $a_{kl}$ are
well-defined provided 
\begin{equation*}
\biggl |\frac{R_{\ell,l}}{\rho_l} - \Id \biggr | \leq r_0\, ,
\end{equation*}
where $r_0$ is the constant of Lemma \ref{l:split}. Recalling the definition of $R_{\ell,l}$ we easily deduce from the maximum principle for transport equations (cf. \eqref{e:max_prin} in Proposition \ref{p:transport_derivatives}) that
$\|\mathring{R}_{\ell,l}\|_0\leq \|\mathring{R}\|_0$.
Hence, from \eqref{e:R_C0_iter} and \eqref{e:est_rhol} we obtain
\[
\biggl |\frac{R_{\ell, l}}{\rho_l} - \Id \biggr | \leq \bar C \frac{\eta}{\min e} \leq 2 \bar C \eta,
\]
for some geometric constant $\bar C$
and thus we will require that 
\begin{equation*}
2 \bar C \eta \leq \frac{r_0}{4}\, .
\end{equation*}

\medskip

The constant $M$ in turn is determined by comparing the estimate \eqref{e:v_C0_iter} for $q+1$ with the definition of the principal perturbation $w_o$ in \eqref{e:def_wo}. Indeed, using \eqref{e:chi_l}-\eqref{e:def_wo} and \eqref{e:est_rhol} we have
$\|w_o\|_{0}\leq \tilde{C} |\Lambda|(\max e)\delta_{q+1}^{\sfrac12}\leq 2 \tilde{C} |\Lambda| \delta_{q+1}^{\sfrac{1}{2}}$ for some geometric constant $\tilde{C}$.
We therefore set 
\begin{equation*}
M=4\tilde{C} |\Lambda|\, ,
\end{equation*}
so that
\begin{equation}\label{e:W_est_0}
\norm{w_o}_0\leq \frac{M}{2} \delta_{q+1}^{\sfrac12}\,.
\end{equation}

\subsection{The operator $\mathcal{R}$ and the Reynolds stress} Following \cite{DS-Inv}, we introduce the following useful operator
to ``invert'' the divergence and define the new Reynolds stress $\mathring{R}_{q+1}$.

\begin{definition}
Let $v$ be a smooth vector field. 
We then define $\RR v$ to be the matrix-valued periodic function
\begin{equation*}
\RR v:=\frac{1}{4}\left(D \P u+(D \P u)^T\right)+\frac{3}{4}\left(D u+(D u)^T\right)-\frac{1}{2}(\div u) \Id,
\end{equation*}
where $u\in C^{\infty}(\T^3,\R^3)$ is the solution of
\begin{equation*}
\Delta u=v-\fint_{\T^3}v\textrm{ in }\T^3
\end{equation*}
with $\fint_{\T^3} u=0$ and $\P$ is the Leray projection onto divergence-free fields with zero average.
\end{definition}

The key point is the following lemma: for its elementary proof we refer the reader to \cite{DS-Inv}. 

\begin{lemma}[$\RR=\textrm{div}^{-1}$]\label{l:reyn}
For any $v\in C^\infty (\T^3, \R^3)$ we have
\begin{itemize}
\item[(a)] $\RR v(x)$ is a symmetric trace-free matrix for each $x\in \T^3$;
\item[(b)] $\div \RR v = v-\fint_{\T^3}v$.
\end{itemize}
\end{lemma}

We next set 
\[
\mathring{R}_{q+1}= R^0+R^1+R^2+R^3+R^4+R^5+R^6,
\]
where
\begin{align}
R^0 &= \mathcal R \left(\partial_tw +v_\ell\cdot \nabla w +w \cdot\nabla v_\ell\right)\label{e:R^0_def}\\
R^1 &=\mathcal R \div \Big(w_o \otimes w_o- \sum_l \chi_l^2 R_{\ell, l} 
-\textstyle{\frac{|w_o|^2}{2}}\Id\Big)\label{e:R^1_def}\\
R^2 &=w_o\otimes w_c+w_c\otimes w_o+w_c\otimes w_c - \textstyle{\frac{|w_c|^2 + 2\langle w_o, w_c\rangle}{3}} {\rm Id}\label{e:R^2_def}\\
R^3 &= w\otimes (v_q - v_\ell) + (v_q-v_\ell)\otimes w
 - \textstyle{\frac{2 \langle (v_q-v_{\ell}), w\rangle}{3}} \Id\label{e:R^3_def}\\
R^4&=\mathring R_q- \mathring{R}_\ell \label{e:R^4_def}\\
R^5&=\sum_l \chi_l^2 (\mathring{R}_{\ell} - \mathring{R}_{l,\ell})\label{e:R^5_def}\\
R^6&=\mathcal R \Big( (- \Delta)^\alpha w  \Big). \label{e:R^6_def}
\end{align}
Observe that $\mathring{R}_{q+1}$ is indeed a traceless symmetric tensor. The corresponding form of the new pressure will then be
\begin{equation}\label{e:def_p_1}
p_{q+1}=p_q-\frac{|w_o|^2}{2} - \frac{1}{3} |w_c|^2 - \frac{2}{3} \langle w_o, w_c\rangle - \frac{2}{3} \langle v_q-v_\ell, w\rangle \,.
\end{equation}

Recalling \eqref{e:mathringRnew} we see that $\sum_l \chi_l^2 \tr R_{\ell, l}$ is a function of time only. Since also 
$\sum_l \chi_l^2 = 1$, it is then straightforward to check that 
\begin{align*}
&\div\mathring{R}_{q+1} - \nabla p_{q+1} -(-\Delta)^\alpha v_{q+1} = \\ &= \partial_t w + \div (v_q\otimes w + w\otimes v_q + w \otimes w) + \div \mathring{R}_q - \nabla p_q\nonumber-(-\Delta)^\alpha v_q\\
&=\partial_t w + \div (v_q\otimes w + w\otimes v_q + w \otimes w) + \partial_t v_q + \div (v_q\otimes v_q)\nonumber\\
&= \partial_t v_{q+1} + \div ( v_{q+1}\otimes v_{q+1})\, .
\end{align*}

The following lemma will play a key role.
\begin{lemma}\label{l:doublesum}
The following identity holds:
\begin{equation}\label{e:doublesum}
w_o\otimes w_o = \sum_l \chi_l^2 R_{\ell, l} + \sum_{(k,l), (k',l'), k\neq - k'} \chi_l \chi_{l'} w_{kl} \otimes w_{k'l'}\, .
\end{equation}
\end{lemma}

\begin{proof}
Recall that the pairs $(k,l)$, $(k',l')$ are chosen so that $k\neq -k'$ if $l$ is even and 
$l'$ is odd. Moreover $\chi_l\chi_{l'} = 0$ if $l$ and $l'$ are distinct and have the same parity. 
Hence the claim follows immediately from our choice of $a_{kl}$ in \eqref{e:a_kl} and Proposition 
\ref{p:Beltrami} and Lemma \ref{l:split} (cf.\  \cite[Proposition 6.1(ii)]{DS-Inv}).
\end{proof}

\subsection{Conditions on the parameters - hierarchy of length-scales}
In the next couple of sections we will need to estimate various expressions involving $v_\ell$ and $w$. To simplify the formulas that we arrive at, we will from now on assume the following conditions on $\mu,\lambda_{q+1}\geq 1$ and $\ell\leq 1$: 

\begin{equation}\label{e:conditions_lambdamu_2}
 \frac{\delta_{q}^{\sfrac{1}{2}}\lambda_q\ell}{\delta_{q+1}^{\sfrac{1}{2}}}\leq1, \quad
\frac{\delta_q^{\sfrac12}\lambda_q}{\mu} +\frac{1}{\ell\lambda_{q+1}}\leq \lambda_{q+1}^{-\beta}\quad\mbox{and}\quad
 \frac{1}{\lambda_{q+1}}\leq  \frac{\delta_{q+1}^{\sfrac12}}{\mu}\, ,
\end{equation}
where $\beta$ is a small positive exponent which will be specified only in the final section. 

These conditions imply the following orderings of length scales, which will be used to simplify the estimates in Section \ref{s:perturbation_estimates}:
\begin{equation}\label{e:ordering_params}
\frac{1}{\delta_{q+1}^{\sfrac{1}{2}} \lambda_{q+1}}\leq \frac{1}{\mu} \leq \frac{1}{\delta_q^{\sfrac{1}{2}}\lambda_{q}}\qquad
\mbox{and}\qquad \frac{1}{\lambda_{q+1}} \leq \ell \leq \frac{1}{\lambda_q}\, . 
\end{equation}

\section{Estimates on the perturbation}\label{s:perturbation_estimates}

The following lemmas are taken directly from \cite{BDIS}, see Lemmas 3.1 and 3.2 therein. A simple inspection of the proof given there show the dependence of
the constants claimed below, which differ slightly from \cite{BDIS} where the same constants are depending upon the energy profile
$e$. Indeed, a simple inspection of the proofs in \cite{BDIS} shows easily that, because of the time discretization introduced in Section 
\ref{s:time-discret} the constants do not depend on the derivatives of $e$, but only on $\min e$ and $\max e$: here we can forget about such dependence because of the assumption $\frac{1}{2} \leq e \leq 1$.

\begin{lemma}\label{l:ugly_lemma}
Assume \eqref{e:conditions_lambdamu_2} holds.  For $t$ in the range $\abs{\mu t-l}<1$ we have
\begin{align}
&\norm{D\Phi_l}_0\leq C\, \label{e:phi_l}\, ,\\
&\norm{D\Phi_l - \Id}_0 \leq C \frac{\delta_q^{\sfrac{1}{2}}\lambda_q}{\mu}\label{e:phi_l_1}\, ,\\
&\norm{D\Phi_l}_N\leq C \frac{\delta_q^{\sfrac{1}{2}} \lambda_q }{\mu \ell^N},& N\ge 1\, .\label{e:Dphi_l_N}
\end{align}
where the constants in \eqref{e:phi_l} and \eqref{e:phi_l_1} are absolute constants, whereas $C$ in \eqref{e:Dphi_l_N} depends only
upon $N$.
Moreover,
\begin{align}
&\norm{a_{kl}}_0+\norm{L_{kl}}_0\leq C \delta_{q+1}^{\sfrac12}\, ,\label{e:L}\\
&\norm{a_{kl}}_N\leq C\delta_{q+1}^{\sfrac12}\lambda_q\ell^{1-N},&N\geq 1\label{e:Da}\\
&\norm{L_{kl}}_N\leq C\delta_{q+1}^{\sfrac12}\ell^{-N},&N\geq 1\label{e:DL}\\
&\norm{\phi_{kl}}_N\leq C \lambda_{q+1} \frac{\delta_q^{\sfrac{1}{2}} \lambda_q}{\mu \ell^{N-1}}
+ C \left(\frac{\delta_q^{\sfrac{1}{2}} \lambda_q \lambda_{q+1}}{\mu}\right)^N\nonumber \\
&\quad\qquad\;\leq C\lambda_{q+1}^{N(1-\beta)}&N\geq 1,\label{e:phi}
\end{align}
where again the constants in \eqref{e:L} and \eqref{e:Da} are absolute and the ones in the other two estimates
depend only upon $N$. 

Consequently, for any $N\geq 0$
\begin{align}
&\norm{w_c}_N \leq C  \delta_{q+1}^{\sfrac12} \frac{\delta_q^{\sfrac12}\lambda_q}{\mu} 
\lambda_{q+1}^N\label{e:corrector_est},\\
&\norm{w_o}_1\leq \frac{M}{2} \delta_{q+1}^{\sfrac12} \lambda_{q+1} +
C \delta_{q+1}^{\sfrac12} \lambda_{q+1}^{1-\beta},\label{e:W_est_1}\\
&\norm{w_o}_N\leq C \delta_{q+1}^{\sfrac12}\lambda_{q+1}^N,\qquad &N\geq 2\label{e:W_est_N}
\end{align}
where the constants $C$s depend only on $N$.
\end{lemma}

\begin{lemma}\label{l:ugly_lemma_2} Recall that $D_t= \partial_t + v_\ell \cdot \nabla$. Under the assumptions of Lemma \ref{l:ugly_lemma} we have
\begin{align}
\|D_t v_\ell\|_N &\leq C \delta_q\lambda_q\ell^{-N}\, , \label{e:Dt_v}\\
\|D_t L_{kl}\|_N &\leq C \delta_{q+1}^{\sfrac{1}{2}} \delta_q^{\sfrac{1}{2}} \lambda_q\ell^{-N}\, ,\label{e:DtL}\\
\|D^2_t L_{kl}\|_N &\leq C\delta_{q+1}^{\sfrac{1}{2}} \delta_q \lambda_q\ell^{-N-1}\, ,\label{e:D2tL}\\
\norm{D_t w_c}_N &\leq C  \delta_{q+1}^{\sfrac12}\delta_q^{\sfrac12}\lambda_q\lambda_{q+1}^N\, ,\label{e:Dt_wc}\\
\norm{D_t w_o}_N &\leq C  \delta_{q+1}^{\sfrac12}\mu\lambda_{q+1}^N\, ,\label{e:Dt_wo}
\end{align}
where the constants depend only upon $N$. 
\end{lemma}

\section{Estimates on the energy}\label{s:estimates_energy}

\begin{lemma}[Estimate on the energy]\label{l:energy}
For any $\varepsilon>0$ we have
\begin{align}
\left|e(t)(1-\delta_{q+2})-\int_{\T^3}|v_{q+1}|^2\,dx\right| &\leq  \frac{\|e'\|_0}{\mu} + 
C \frac{\delta_{q+1}\delta_q^{\sfrac{1}{2}}\lambda_q}{\mu}  \nonumber \\ & +C\frac{\delta_{q+1}^{\sfrac{1}{2}} \delta_q^{\sfrac{1}{2}}\lambda_q}{\lambda_{q+1}} +C \frac{\lambda_q^{2\alpha+\varepsilon} \delta_q^{\alpha+\varepsilon}}{\mu}\, ,\label{e:energy}
\end{align}
where $C$ is an absolute constant.
\end{lemma}

The proof of Lemma \ref{l:energy} is similar to the one of the analogous Lemma 4.1 in \cite{BDIS}. However we include the proof for the reader's convenience because:
\begin{itemize}
\item the additional dissipative term alters the argument at a certain point;
\item we need the specific dependence of the estimates upon the energy profile $e$, which in \cite{BDIS} is not taken into account.
\end{itemize}

\begin{proof} Define
$$
\bar{e}(t):=3(2\pi)^3\sum_l\chi_l^2(t)\rho_l.
$$
Using Lemma \ref{l:doublesum} we then have
\begin{align}
|w_o|^2 &=\sum_l \chi_l^2\tr R_{\ell,l} + \sum_{(k,l), (k', l'), k\neq -k^{\prime}} \chi_l\chi_{l^{\prime}} w_{kl} \cdot w_{k,l'}\nonumber\\
&= (2\pi)^{-3}\bar{e} + \sum_{(k,l), (k',l'), k\neq - k'} \chi_l \chi_{k'} a_{kl} a_{k'l'} \phi_{kl} \phi_{k'l'} e^{i\lambda_{q+1} (k+k') \cdot x}\, .
\end{align}
Observe that $\bar{e}$ is a function of $t$ only and that, since $(k+k')\neq 0$ in the sum above, we can apply
Proposition \ref{p:stat_phase}(i) with $m=1$. From Lemma \ref{l:ugly_lemma} we then deduce
\begin{align}
& \left|\int_{\T^3}|w_o|^2\,dx-\bar{e} (t)\right|\leq C\frac{\delta_{q+1} \delta_q^{\sfrac{1}{2}} \lambda_q}{\mu}  +  C \frac{\delta_{q+1}\lambda_q}{\lambda_{q+1}} \, .\label{e:energy1}
\end{align}
Next we recall \eqref{e:w_compact_form}, integrate by parts and use \eqref{e:L} and \eqref{e:phi} to reach
\begin{equation}\label{e:energy2}
\left|\int_{\T^3}v_{q+1}\cdot w\,dx\right|\leq C \frac{\delta_{q+1}^{\sfrac{1}{2}} \delta_q^{\sfrac{1}{2}}\lambda_q}{\lambda_{q+1}}\, .
\end{equation}
Note also that by \eqref{e:corrector_est} we have
\begin{align}
\int_{\T^3}|w_c|^2+\abs{w_c w_o}\,dx&\leq C\frac{\delta_{q+1} \delta_q^{\sfrac{1}{2}} \lambda_q}{\mu}\,  .\label{e:energy3}
\end{align}
Summarizing, so far we have achieved
\begin{align}
&\left|\int_{\T^3} |v_{q+1}|^2 \,dx- \left(\bar{e} (t) +\int_{\T^3} |v_q|^2\,dx\right)\right| \stackrel{\eqref{e:energy2}}{\leq} 
\left|\int_{\T^3} |w|^2\,dx - \bar{e} (t)\right| + C \frac{\delta_{q+1}^{\sfrac{1}{2}} \delta_q^{\sfrac{1}{2}}\lambda_q}{\lambda_{q+1}}\nonumber\\
&\stackrel{\eqref{e:energy3}}{\leq} \left|\int_{\T^3} |w_o|^2 \,dx- \bar{e} (t)\right| + C \frac{\delta_{q+1}^{\sfrac{1}{2}} \delta_q^{\sfrac{1}{2}}\lambda_q}{\lambda_{q+1}} + C\frac{\delta_{q+1} \delta_q^{\sfrac{1}{2}} \lambda_q}{\mu}\nonumber\\
&\stackrel{\eqref{e:energy1}}{\leq} C \frac{\delta_{q+1}^{\sfrac{1}{2}} \delta_q^{\sfrac{1}{2}}\lambda_q}{\lambda_{q+1}} + C\frac{\delta_{q+1} \delta_q^{\sfrac{1}{2}} \lambda_q}{\mu}\,.\label{e:energy10}
\end{align}

Next, recall that
\begin{align*}
\bar{e} (t) &= 3 (2\pi)^3\sum_l \chi_l^2  \rho_l\\
&=(1-\delta_{q+2})\sum_l \chi_l^2 e\left(\frac{\mu}{l}\right)-\sum_l\chi_l^2\int_{\T^3}|v_q(x,l\mu^{-1})|^2\,dx\, .
\end{align*}
Since $\abs{t-\frac{l}{\mu}}<\mu^{-1}$ on the support of $\chi_l$ and since $\sum_l \chi_l^2 =1$, we have
\[
\left|e (t) - \sum_l \chi_l^2 e\left(\frac{l}{\mu}\right)\right| \leq \|e'\|_0 \mu^{-1}\, .
\]
Moreover, using the Navier Stokes-Reynolds equation, we can compute
\begin{align}
&\int_{\T^3} \left(|v_q(x,t)|^2-\left|v_q\left(x,l\mu^{-1}\right)\right|^2\right)\, dx = \int_{\frac{l}{\mu}}^t \int_{\T^3} \partial_t |v_q|^2\nonumber\\
&=- \int_{\frac{l}{\mu}}^t \int_{\T^3} {\rm div}\, \left(v_q \left(|v_q|^2 + 2p_q\right)\right) + 2 \int_{\frac{l}{\mu}}^t \int_{\T^3} v_q \cdot \div \mathring{R}_q -2\int_{\frac{l}{\mu}}^t \int_{\mathbb{T}^3}|(-\Delta)^{\frac{\alpha}{2}} v_q|^2 \nonumber \\
&= -   2 \int_{\frac{l}{\mu}}^t \int_{\T^3} Dv_q: \mathring{R}_q -2\int_{\frac{l}{\mu}}^t \int_{\mathbb{T}^3}|(-\Delta)^{\frac{\alpha}{2}} v_q|^2 .\nonumber
\end{align}
Thus, for $\left|t-\frac{l}{\mu}\right|\leq \mu^{-1}$ we conclude
\[
\left|\int_{\T^3} |v_q(x,t)|^2-\left|v_q(x,l\mu^{-1})\right|^2\, dx \right| \leq C \frac{\delta_{q+1}\delta_q^{\sfrac{1}{2}}\lambda_q}{\mu} +C \frac{\lambda_q^{2\alpha+\varepsilon} \delta_q^{\alpha+\varepsilon}}{\mu} .
\]
Using again $\sum \chi_l^2=1$, we then conclude
\begin{equation}\label{e:energy4}
\left|e(t) (1-\delta_{q+2}) - \left(\bar{e} (t)+ \int_{\T^3} |v_q(x,t)|^2 \, dx\right)\right| \leq \frac{\|e'\|_0}{\mu} + 
C \frac{\delta_{q+1}\delta_q^{\sfrac{1}{2}}\lambda_q}{\mu} +C \frac{\lambda_q^{2\alpha+\varepsilon} \delta_q^{\alpha+\varepsilon}}{\mu}\, .
\end{equation}
The desired conclusion \eqref{e:energy} follows from \eqref{e:energy10} and \eqref{e:energy4}, indeed by triangular inequality
\begin{align*}
&\left|e(t)(1-\delta_{q+2})-\int_{\T^3}|v_{q+1}|^2\,dx\right| \leq \left|e(t) (1-\delta_{q+2}) - \left(\bar{e} (t)+ \int_{\T^3} |v_q(x,t)|^2 \, dx\right)\right|\\
&+ \left|\int_{\T^3} |v_{q+1}|^2 \,dx- \left(\bar{e} (t) +\int_{\T^3} |v_q|^2\,dx\right)\right|\leq  \frac{\| e'\|_0}{\mu} + 
C \frac{\delta_{q+1}\delta_q^{\sfrac{1}{2}}\lambda_q}{\mu}  \nonumber \\ & +C\frac{\delta_{q+1}^{\sfrac{1}{2}} \delta_q^{\sfrac{1}{2}}\lambda_q}{\lambda_{q+1}} +C \frac{\lambda_q^{2\alpha+\varepsilon} \delta_q^{\alpha+\varepsilon}}{\mu}.
\end{align*}
\end{proof}

\section{Estimates on the Reynolds stress}\label{s:reynolds}

In this section we bound the new Reynolds Stress $\mathring R_{q+1}$. The bounds for the tensors $R^0, \ldots, R^5$ are
essentially the same as in \cite{BDIS}, with the only exception that we have kept track of the dependence of the constants appearing in the estimates. 

Recalling the definition $ R^6=\mathcal R ( (- \Delta)^\alpha w )$ of the dissipative part of the error we can easly guess why we have the restriction  $\alpha \in(0,\sfrac{1}{2})$ with the following heuristic argument.
The oscillations of the map $w \sim \delta_{q+1}^{\sfrac{1}{2}}$ are driven by the parameter $\lambda_{q+1}$. The two operators $\mathcal R$ and $(- \Delta)^\alpha$ are differentials operators of order $-1$ and $2\alpha$ respectively, so the \textit{heuristic} gives us  $ R^6 \sim \delta_{q+1}^{\sfrac{1}{2}} \lambda_{q+1}^{2\alpha-1}$, so that if $\alpha < \sfrac{1}{2}$ we can make 
$ R^6 \sim \delta_{q+2}$, which is the condition required for our inductive scheme.

\begin{proposition}\label{p:R} For any choice of small positive numbers $\eps$ and $\beta$, there is a constant $C$ (depending only upon the latter parameters) such that, if $\mu$, $\lambda_{q+1}$ and $\ell$ satisfy the conditions \eqref{e:conditions_lambdamu_2}, then we have
\begin{align}
\|R^0\|_0 +\frac{1}{\lambda_{q+1}}\|R^0\|_1+\frac{1}{\mu}\|D_tR^0\|_0&\leq C\frac{\delta_{q+1}^{\sfrac{1}{2}} \mu}{\lambda_{q+1}^{1-\eps}}+\frac{\delta_{q+1}^{\sfrac12}\delta_q\lambda_q}{\lambda_{q+1}^{1-\eps}\mu\ell}\, ,\label{e:R0}\\
\|R^1\|_0 +\frac{1}{\lambda_{q+1}}\|R^1\|_1+\frac{1}{\mu}\|D_tR^1\|_0&\leq C \frac{\delta_{q+1}\delta_q^{\sfrac{1}{2}} \lambda_q \lambda_{q+1}^\eps}{\mu}\, ,\label{e:R1}\\
\|R^2\|_0 +\frac{1}{\lambda_{q+1}}\|R^2\|_1+\frac{1}{\mu}\|D_tR^2\|_0&\leq C \frac{\delta_{q+1}\delta_q^{\sfrac{1}{2}} \lambda_q}{\mu}\, , \label{e:R2}\\
\|R^3\|_0 +\frac{1}{\lambda_{q+1}}\|R^3\|_1+\frac{1}{\mu}\|D_tR^3\|_0&\leq C \delta_{q+1}^{\sfrac{1}{2}} \delta_q^{\sfrac{1}{2}} \lambda_q \ell\, ,\label{e:R3}\\
\|R^4\|_0 +\frac{1}{\lambda_{q+1}}\|R^4\|_1+\frac{1}{\mu}\|D_tR^4\|_0&\leq C \frac{\delta_{q+1} \delta_q^{\sfrac{1}{2}}\lambda_q}{\mu} + C\delta_{q+1} \lambda_q \ell\, ,\label{e:R4}\\
\|R^5\|_0 +\frac{1}{\lambda_{q+1}}\|R^5\|_1+\frac{1}{\mu}\|D_tR^5\|_0&\leq C \frac{\delta_{q+1} \delta_q^{\sfrac{1}{2}}\lambda_q}{\mu}\, , \label{e:R5} \\
\|R^6\|_0 +\frac{1}{\lambda_{q+1}}\|R^6\|_1+\frac{1}{\mu}\|D_tR^6\|_0&\leq   C \frac{\delta_{q+1}^{\sfrac{1}{2}}\lambda_{q+1}^{2\alpha+\varepsilon}}{\mu}\, .\label{e:R6}
\end{align}
Thus
\begin{equation}
\begin{split}
&\|\mathring{R}_{q+1}\|_0+\frac{1}{\lambda_{q+1}}\|\mathring{R}_{q+1}\|_1 +\frac{1}{\mu}\|D_t\mathring{R}_{q+1}\|_0\leq \\
&\leq C \left(\frac{\delta_{q+1}^{\sfrac{1}{2}} \mu}{\lambda_{q+1}^{1-\eps}} + \frac{\delta_{q+1} \delta_q^{\sfrac{1}{2}} \lambda_q \lambda_{q+1}^\eps}{\mu} +
\delta_{q+1}^{\sfrac{1}{2}} \delta_q^{\sfrac{1}{2}} \lambda_q \ell+\frac{\delta_{q+1}^{\sfrac12}\delta_q\lambda_q}{\lambda_{q+1}^{1-\eps}\mu\ell} +C \frac{\delta_{q+1}^{\sfrac{1}{2}}\lambda_{q+1}^{2\alpha+\varepsilon}}{\mu}\right)\, ,\label{e:allR}
\end{split}
\end{equation}
and, moreover,
\begin{multline}
\|\partial_t \mathring{R}_{q+1} + v_{q+1}\cdot \nabla \mathring{R}_{q+1}\|_0\leq \\
\leq C \delta_{q+1}^{\sfrac{1}{2}}\lambda_{q+1} \left(\frac{\delta_{q+1}^{\sfrac{1}{2}} \mu}{\lambda_{q+1}^{1-\eps}} 
+ \frac{\delta_{q+1} \delta_q^{\sfrac{1}{2}} \lambda_q \lambda_{q+1}^\eps}{\mu} 
+ \delta_{q+1}^{\sfrac{1}{2}} \delta_q^{\sfrac{1}{2}} \lambda_q \ell+\frac{\delta_{q+1}^{\sfrac12}\delta_q\lambda_q}{\lambda_{q+1}^{1-\eps}\mu\ell}+C \frac{\delta_{q+1}^{\sfrac{1}{2}}\lambda_{q+1}^{2\alpha+\varepsilon}}{\mu} \right).\label{e:Dt_R_all}
\end{multline}
As in the previous sections, all the constants $C$ appearing in the estimates are absolute constants.
\end{proposition}

\begin{proof} The arguments for the estimates \eqref{e:R0}. \eqref{e:R1}, \eqref{e:R2}, \eqref{e:R3}, \eqref{e:R4} and \eqref{e:R5} are 
the same as those of \cite{BDIS} for the same estimates claimed in Proposition 5.1 therein. We therefore give the proof only for the remaining ones.

\medskip

\noindent{\bf Estimates on $R^6$.} Since $(-\Delta)^\alpha$ and the operator $\RR$ commute, the idea is to obtain an estimate for both $\| \RR w\|_0$ and $\|\RR w\|_1$ and then interpolate
\begin{equation}
\| (-\Delta)^\alpha \RR w \|_0  \leq C [\RR w]_{2\alpha+\varepsilon} \leq C \| \RR w\|_0^{1-2\alpha-\varepsilon}\|\RR w\|_1^{2\alpha+\varepsilon}.
\end{equation}
Remember that 
\[
w=\sum_{k,l} \chi_l L_{kl} \phi_{kl} e^{i \lambda_{q+1}k \cdot x}=\sum_{k,l} B_{kl} e^{i \lambda_{q+1}k \cdot x}.
\]
First of all observe that
\begin{equation}
[B_{kl}]_N \leq  C( \delta_{q+1}^{\sfrac{1}{2}} \ell^{-N}+\delta_{q+1}^{\sfrac{1}{2}} \lambda_{q+1}^{N(1-\beta)}) \leq C \delta_{q+1}^{\sfrac{1}{2}} \lambda_{q+1}^{N(1-\beta)},
\end{equation}
thus from Proposition \ref{p:stat_phase} we get, choosing $N$ such that $N \beta \geq 1$
\begin{multline}
\| \RR w\|_0 \leq C \bigg( \frac{1}{\lambda_{q+1}^{1-\varepsilon}} \|B_{kl} \|_0 + \frac{1}{\lambda_{q+1}^{N-\varepsilon}}[B_{kl}]_N +\frac{1}{\lambda_{q+1}^{N}}[B_{kl}]_{N+\varepsilon} \bigg) \\
\leq C\bigg( \frac{\delta_{q+1}^{\sfrac{1}{2}}}{\lambda_{q+1}^{1-\varepsilon}}+ \delta_{q+1}^{\sfrac{1}{2}} \lambda_{q+1}^{\varepsilon -N\beta} +\delta_{q+1}^{\sfrac{1}{2}} \lambda_{q+1}^{\varepsilon -N\beta -\varepsilon \beta} \bigg) \leq C \delta_{q+1}^{\sfrac{1}{2}} \lambda_{q+1}^{\varepsilon -1}. \label{estC0Rw}
\end{multline}
Analogously we get 
\begin{equation}\label{estC1Rw}
\| \RR w \|_1 \leq C \delta_{q+1}^{\sfrac{1}{2}} \lambda_{q+1}^{\varepsilon}.
\end{equation}
Combining \eqref{estC0Rw} and \eqref{estC1Rw}, by interpolation we get
\begin{equation}
\| (-\Delta)^\alpha \RR w \|_0 \leq C \delta_{q+1}^{\sfrac{1}{2}} \lambda_{q+1}^{2\alpha -1+\varepsilon}.
\end{equation}
Similarly, with analogous estimates on $\|\RR w\|_2$ and again by interpolation we easily conclude
\begin{equation}
\| (-\Delta)^\alpha \RR w \|_1 \leq C \delta_{q+1}^{\sfrac{1}{2}} \lambda_{q+1}^{2\alpha +\varepsilon}
\end{equation}

\medskip

\noindent{\bf Estimates on $D_t R^6$.} Since $(-\Delta)^\alpha$ and $D_t= \partial_t+v_\ell \cdot \nabla$ do not commute, we have
\begin{equation}
D_t (-\Delta)^\alpha \RR w= (-\Delta)^\alpha D_t \RR w - [(-\Delta)^\alpha , D_t] \RR w.
\end{equation}

For the first term $(-\Delta)^\alpha D_t \RR w$ we proceed, as usually, by estimate  $D_t \RR $ in $C^0$ and in $C^1$ and using interpolation.
First we compute
\[
D_t w= \sum_{kl} (\partial_t \chi_l L_{kl}+ \chi_l D_t L_{kl}) \phi_{kl} e^{i \lambda_{q+1}k \cdot x}:= \sum_{kl} B_{kl}^{'}e^{i \lambda_{q+1}k \cdot x}.
\]
We have that 
\begin{equation}
\| B_{kl}^{'}\|_0 \leq \mu \|L_{kl}\|_0 +\|D_t L_{kl}\|_0 \leq C(\mu \delta_{q+1}^{\sfrac{1}{2}} + \delta_{q+1}^{\sfrac{1}{2}} \delta_{q}^{\sfrac{1}{2}} \lambda_q ) \leq C  \mu \delta_{q+1}^{\sfrac{1}{2}},
\end{equation}
and similarly, $\forall N \geq 1$
\begin{equation}
\| B_{kl}^{'}\|_N \leq C \mu \delta_{q+1}^{\sfrac{1}{2}} \lambda_{q+1}^{N(1-\beta)}.
\end{equation}
Using the usual commutator structures we have
\[
D_t \RR w = \sum_{kl} [v_{\ell}, \RR]\bigg(D B_{kl} e^{i \lambda_{q+1}k \cdot x}\bigg) + i \lambda_{q+1} [v_{\ell} \cdot k, \RR]\bigg( B_{kl} e^{i \lambda_{q+1}k \cdot x}\bigg) +\RR \bigg( B_{kl}^{'} e^{i \lambda_{q+1}k \cdot x}\bigg).
\]
 By  Proposition \ref{p:stat_phase} and choosing $N$ sufficiently large we have
\begin{equation}
\left\| \RR \bigg( B_{kl}^{'} e^{i \lambda_{q+1}k \cdot x}\bigg) \right\|_0 \leq C \bigg(\frac{\|B_{kl}^{'}\|_0}{\lambda_{q+1}^{1-\varepsilon}}+\frac{\|B_{kl}^{'}\|_N}{\lambda_{q+1}^{N-\varepsilon}}+\frac{\|B_{kl}^{'}\|_{N+\varepsilon}}{\lambda_{q+1}^{N}} \bigg) \leq  C \mu \delta_{q+1}^{\sfrac{1}{2}} \lambda_{q+1}^{-1},\label{termine1}
\end{equation}
and moreover from Proposition \ref{p:commutator}
\begin{multline}
\left\|  [v_{\ell}, \RR]\bigg(D B_{kl} e^{i \lambda_{q+1}k \cdot x}\bigg)\right\|_0 \leq C\lambda_{q+1}^{\varepsilon-2} \|D B_{kl} \|_0 \|v_\ell\|_1 \\
+C \lambda_{q+1}^{\varepsilon-N} \bigg( \| B_{kl} \|_{N+\varepsilon} \|v_\ell\|_{1+\varepsilon} + \| B_{kl} \|_{1+\varepsilon} \|v_\ell\|_{N+\varepsilon} \bigg) \leq  C \mu \delta_{q+1}^{\sfrac{1}{2}} \lambda_{q+1}^{-1}.\label{termine2}
\end{multline}
It is not difficult to see that the same estimate holds also for the last term, i.e.
\begin{equation}
 \left\| i \lambda_{q+1} [v_{\ell} \cdot k, \RR]\bigg( B_{kl} e^{i \lambda_{q+1}k \cdot x}\bigg)\right\|_0 \leq C \mu \delta_{q+1}^{\sfrac{1}{2}} \lambda_{q+1}^{-1}, \label{termine3}
\end{equation} indeed we do not have any derivatives on $B_{kl}$ but we have an extra factor $\lambda_{q+1}$.
Thanks to \eqref{termine1}, \eqref{termine2} and \eqref{termine3} we conclude 
\begin{equation}\label{DtR6C0}
\| D_t \RR w\|_0 \leq C \mu \delta_{q+1}^{\sfrac{1}{2}} \lambda_{q+1}^{-1},
\end{equation}
and, analogously, 
\begin{equation}\label{DtR6C1}
\| D_t \RR w\|_1 \leq C \mu \delta_{q+1}^{\sfrac{1}{2}}.
\end{equation}
Thus from \eqref{DtR6C0}, \eqref{DtR6C1} and by interpolation we conclude

\begin{equation}\label{laplDtR6}
\left\| (-\Delta)^\alpha D_t \RR w \right\|_0 \leq C \mu \delta_{q+1}^{\sfrac{1}{2}} \lambda_{q+1}^{2\alpha +\varepsilon-1}.
\end{equation}
It remains to estimate $[(-\Delta)^\alpha , D_t] \RR w$. Notice that in this commutator we have the obvious cancellation of the time derivative term, so

\begin{multline}
\left\| [(-\Delta)^\alpha , D_t] \RR w \right\|_0 = \left\|  (-\Delta)^\alpha (v_\ell \cdot \nabla ) \RR w- v_\ell \cdot \nabla  (-\Delta)^\alpha \RR w \right\|_0 \\
\leq C \| (v_\ell \cdot \nabla ) \RR w \|_0 ^ {1-2\alpha-\varepsilon}[ (v_\ell \cdot \nabla ) \RR w ]_1 ^ {2\alpha +\varepsilon} + \|v_\ell  \|_0 \| (-\Delta)^\alpha \RR w\|_1 \\
\leq C \| (v_\ell \cdot \nabla ) \RR w \|_0 ^ {1-2\alpha-\varepsilon}[ (v_\ell \cdot \nabla ) \RR w ]_1 ^ {2\alpha +\varepsilon} + M \delta_{q+1}^{\sfrac{1}{2}} \lambda_{q+1}^{2\alpha +\varepsilon},
\end{multline}
but since 

\[
\| (v_\ell \cdot \nabla ) \RR w \|_0 \leq \| v_\ell \|_0 \| \RR w\|_1 \leq C \delta_{q+1}^{\sfrac{1}{2}} \lambda_{q+1}^{\varepsilon},
\] 
and
\[
[ (v_\ell \cdot \nabla ) \RR w ]_1 \leq [v_\ell]_1 \| \RR w\|_1 + \| \RR w \|_2 \leq C\delta_{q}^{\sfrac{1}{2}} \lambda_{q}\delta_{q+1}^{\sfrac{1}{2}} \lambda_{q+1}^{\varepsilon}+ C\delta_{q+1}^{\sfrac{1}{2}} \lambda_{q+1}^{\varepsilon+1} \leq \delta_{q+1}^{\sfrac{1}{2}} \lambda_{q+1}^{\varepsilon+1},
\]
we have that 
\begin{equation}\label{Dtcommut}
\left\| [(-\Delta)^\alpha , D_t] \RR w \right\|_0 \leq C \delta_{q+1}^{\sfrac{1}{2}} \lambda_{q+1}^{2\alpha +\varepsilon}
\end{equation}
Putting together \eqref{laplDtR6} and \eqref{Dtcommut}, since $\mu \leq \lambda_{q+1}$, we finally conclude
\[
\|D_t R^6\|_0 \leq C \delta_{q+1}^{\sfrac{1}{2}} \lambda_{q+1}^{2\alpha +\varepsilon}.
\]
Now \eqref{e:R6} follows since, again, by our choice of the parameter, $\frac{1}{\lambda_{q+1}} \leq \frac{1}{\mu}$.
\medskip
\begin{remark}
In estimating $[(-\Delta)^\alpha , D_t] \RR w$ we did not exploit its commutator nature since an improvement of the \textit{coarse} estimate above would not lead to a better result (in fact the term $\frac{\lambda_q^{2\alpha+\varepsilon} \delta_q^{\alpha+\varepsilon}}{\mu}$ in the energy estimate obstructs the usefulness of any better bound on $R^6$).
\end{remark}
\medskip
\noindent{\bf Conclusion.} \eqref{e:allR} is an obvious consequence of the estimates for the terms $R^0, R^1, \ldots , R^6$. To achieve \eqref{e:Dt_R_all}, observe that
\[
\|\partial_t \mathring{R}_{q+1} + v_{q+1}\cdot \nabla \mathring{R}_{q+1}\|_0 \leq \|D_t \mathring{R}_1\|_0 + \left(\|v_{q+1} -v_\ell\|_0 + \|w\|_0\right)\|\mathring{R}_{q+1}\|_1 \, .
\]
 On the other hand, by \eqref{e:v-v_ell} and \eqref{e:conditions_lambdamu_2}, $\|v_{q+1} -v_\ell\|_0 \leq C \delta_q^{\sfrac{1}{2}} \lambda_q \ell\leq \delta_{q+1}^{\sfrac{1}{2}}$. Moreover, by \eqref{e:W_est_0}, \eqref{e:corrector_est} and \eqref{e:ordering_params} 
$\|w\|\leq \|w_o\|_0 + \|w_c\|_0 \leq  C \delta_{q+1}^{\sfrac{1}{2}}$. Thus, by \eqref{e:allR} we conclude
\begin{align*}
& \|\partial_t \mathring{R}_{q+1} + v_{q+1}\cdot \nabla \mathring{R}_{q+1}\|_0 \leq
C \left(\mu + \delta_{q+1}^{\sfrac{1}{2}}\lambda_{q+1}\right)\\
&\qquad 
\left(\frac{\delta_{q+1}^{\sfrac{1}{2}} \mu}{\lambda_{q+1}^{1-\eps}} + \frac{\delta_{q+1} \delta_q^{\sfrac{1}{2}} \lambda_q \lambda_{q+1}^\eps}{\mu} +
\delta_{q+1}^{\sfrac{1}{2}} \delta_q^{\sfrac{1}{2}} \lambda_q \ell+\frac{\delta_{q+1}^{\sfrac12}\delta_q\lambda_q}{\lambda_{q+1}^{1-\eps}\mu\ell} +C \frac{\delta_{q+1}^{\sfrac{1}{2}}\lambda_{q+1}^{2\alpha+\varepsilon}}{\mu} \right)
\end{align*}
Since by \eqref{e:ordering_params} $\mu \leq \delta_{q+1}^{\sfrac{1}{2}}\lambda_{q+1}$, \eqref{e:Dt_R_all}
follows easily.
\end{proof}

\section{Proofs of Proposition \ref{p:iterate} and of Proposition \ref{p:tecnica}}

\subsection{Choice of the parameters $\mu$ and $\ell$}
In order to proceed, recall that the sequences $\{\delta_q\}_{q\in\N}$ and $\{\lambda_q\}_{q\in\N}$ 
are chosen to satisfy
\[
\delta_q=a^{-b^q},\quad a^{cb^{q+1}}\leq \lambda_q \leq 2 a^{cb^{q+1}}
\]
for some given constants $c>\max(\frac{5}{2},\frac{3-2\alpha}{2(1-2\alpha)})$ and $b>1$ and for $a>1$. 
Note that this has the consequence that if $a$ is chosen sufficiently large (depending only on $b>1$) then 
\begin{align}\label{e:condition_deltalambda}
\delta_q^{\sfrac12}\lambda_q^{\sfrac15}\leq 
\delta_{q+1}^{\sfrac12}\lambda_{q+1}^{\sfrac15},\quad\delta_{q+1}\leq\delta_q,\quad\textrm{ and }\quad\lambda_q\leq \lambda_{q+1}^{\frac{2}{b+1}}\,.
\end{align}
\medskip

We start by specifying the parameters $\mu=\mu_q$ and $\ell=\ell_q$: we determine them
optimizing the right hand side of
\eqref{e:allR}. More precisely, we set 
\begin{equation}\label{e:choice_mu}
\mu := \delta_{q+1}^{\sfrac{1}{4}} \delta_q^{\sfrac{1}{4}} \lambda_q^{\sfrac{1}{2}} \lambda_{q+1}^{\sfrac{1}{2}}
\end{equation}
so that the first two expressions in \eqref{e:allR} are equal, and then, having determined $\mu$, set
\begin{equation}\label{e:choice_ell}
\ell := \delta_{q+1}^{-\sfrac{1}{8}}\delta_{q}^{\sfrac{1}{8}}\lambda_{q}^{-\sfrac{1}{4}}\lambda_{q+1}^{-\sfrac{3}{4}}
\end{equation}
so that the third and fourth expressions in \eqref{e:allR} are equal (up to a factor $\lambda_{q+1}^\eps$). 
In turn, these choices lead to 
\begin{align}
\|\mathring{R}_{q+1}\|_0+\frac{1}{\lambda_{q+1}}\|\mathring{R}_{q+1}\|_1
&\leq C\delta_{q+1}^{\sfrac34}\delta_q^{\sfrac14}\lambda_q^{\sfrac12}\lambda_{q+1}^{\eps-\sfrac12}+
C\delta_{q+1}^{\sfrac38}\delta_q^{\sfrac58}\lambda_q^{\sfrac34}\lambda_{q+1}^{\eps-\sfrac34} +C \frac{\delta_{q+1}^{\sfrac{1}{2}}\lambda_{q+1}^{2\alpha+\varepsilon}}{\mu} \nonumber\\
&=C\delta_{q+1}^{\sfrac34}\delta_q^{\sfrac14}\lambda_q^{\sfrac12}\lambda_{q+1}^{\eps-\sfrac12}
\left(1+\left(\frac{\delta_q^{\sfrac12}\lambda_q^{\sfrac13}}{\delta_{q+1}^{\sfrac12}\lambda_{q+1}^{\sfrac13}}\right)^{\sfrac34}\right) +C \frac{\delta_{q+1}^{\sfrac{1}{2}}\lambda_{q+1}^{2\alpha+\varepsilon}}{\mu}\nonumber\\
&\stackrel{\eqref{e:condition_deltalambda}}{\leq} C\delta_{q+1}^{\sfrac34}\delta_q^{\sfrac14}\lambda_q^{\sfrac12}\lambda_{q+1}^{\eps-\sfrac12}+C \frac{\delta_{q+1}^{\sfrac{1}{2}}\lambda_{q+1}^{2\alpha+\varepsilon}}{\mu}\, .\label{e:Rq+1}
\end{align}
Observe also that by \eqref{e:Dt_R_all}, we have
\begin{equation}\label{e:Rq+1_adv}
\|\partial_t \mathring{R}_{q+1} + v_{q+1} \cdot \nabla \mathring{R}_{q+1}\|_0
\leq C \delta_{q+1}^{\sfrac{1}{2}} \lambda_{q+1} \left(\delta_{q+1}^{\sfrac34}\delta_q^{\sfrac14}\lambda_q^{\sfrac12}\lambda_{q+1}^{\eps-\sfrac12} +\frac{\delta_{q+1}^{\sfrac{1}{2}}\lambda_{q+1}^{2\alpha+\varepsilon}}{\mu} \right).
\end{equation}
Let us check that the conditions \eqref{e:conditions_lambdamu_2} are satisfied for some $\beta>0$ (remember that $\beta$ should be independent of $q$).
To this end we calculate
\begin{align*}
\frac{\delta_q^{\sfrac12}\lambda_q\ell}{\delta_{q+1}^{\sfrac12}}&=\left(\frac{\delta_q^{\sfrac12}\lambda_q^{\sfrac35}}{\delta_{q+1}^{\sfrac12}\lambda_{q+1}^{\sfrac35}}\right)^{\sfrac54}\,,\qquad
\frac{\delta_q^{\sfrac12}\lambda_q}{\mu}=\left(\frac{\delta_q^{\sfrac12}\lambda_q}{\delta_{q+1}^{\sfrac12}\lambda_{q+1}}\right)^{\sfrac12}\,,\\
\frac{1}{\ell\lambda_{q+1}}&=\left(\frac{{\delta_{q+1}^{\sfrac12}}\lambda_q}
{{ \delta_{q}^{\sfrac12}}\lambda_{q+1}}\right)^{\sfrac14}\,,\qquad 
\frac{\mu}{\delta_{q+1}^{\sfrac12}\lambda_{q+1}}=\left(\frac{\delta_q^{\sfrac12}\lambda_q}{\delta_{q+1}^{\sfrac12}\lambda_{q+1}}\right)^{\sfrac12}\,.
\end{align*}
Hence the conditions \eqref{e:conditions_lambdamu_2} follow from
\eqref{e:condition_deltalambda} choosing $\beta=\frac{b-1}{5b+5}$.

\subsection{Proof of Proposition \ref{p:iterate}}\label{s:hoeld_exp}
Recall that $c>\max\{\frac{5}{2},\frac{3-2\alpha}{2(1-2\alpha)}\}$ and $b>1$. We also
keep the small positive parameter $\eps>0$ whose choice will be specified later. 
The proposition is proved inductively. The initial triple is defined to be the triple $(v_0, p_0, \mathring{R}_0)$ derived in Lemma \ref{l:partenza}
(observe that, since $(c-1) b -\frac{1}{2} >1$, \eqref{e:condizioni_su_abc_2} is stronger than \eqref{e:condizioni_su_abc}). 
Given now $(v_q, p_q, \mathring{R}_q)$ satisfying the estimates \eqref{e:v_C0_iter}-\eqref{e:R_Dt_iter},
we claim that the triple $(v_{q+1}, p_{q+1}, \mathring{R}_{q+1})$ constructed above  
satisfies again all the corresponding estimates.

\smallskip

\noindent{\bf Estimates on $\mathring{R}_{q+1}$.}
Note first of all that, using the form of the estimates in \eqref{e:allR} and \eqref{e:Dt_R_all}, the estimates \eqref{e:R_C1_iter} and \eqref{e:R_Dt_iter} follow from \eqref{e:R_C0_iter}. On the other hand, in light of \eqref{e:Rq+1}, \eqref{e:R_C0_iter} follows from the recursion relations
$$
C\delta_{q+1}^{\sfrac34}\delta_q^{\sfrac14}\lambda_q^{\sfrac12}\lambda_{q+1}^{\eps-\sfrac12}\leq \frac{\eta}{2}\delta_{q+2},
$$

$$
C \frac{\delta_{q+1}^{\sfrac{1}{2}}\lambda_{q+1}^{2\alpha+\varepsilon}}{\mu}\leq \frac{\eta}{2}\delta_{q+2}.
$$
Using our choice of $\delta_q$ and $\lambda_q$ from Proposition \ref{p:iterate}, we see that the first inequality is equivalent to
$$
C\leq a^{\frac{1}{4}b^q(1+3b-2cb+(2c-4-4\eps c)b^2)},
$$
which, since $b>1$, is satisfied for all $q\geq 1$ for a sufficiently large fixed constant $a>1$, provided 
$$
\left(1+3b-2cb+(2c-4-4\eps c)b^2\right)>0.
$$
Factorizing, we obtain the inequality $(b-1)((2c-4) b-1)-4\eps cb^2>0$. It is then easy to see that for any $b>1$ and $c>5/2$ there exists
$\eps>0$ so that this inequality is satisfied. In this way we can choose $\eps>0$ (and $\beta$ above) depending solely on $b$ and $c$. Regarding the second recursion relation, it is equivalent to
\[
C\leq a^{b^q(-b^2+\frac{1}{4}b-(2\alpha+\varepsilon-\frac{1}{2})cb^2-\frac{1}{4}+\frac{1}{2}cb)},
\]
Thus, evaluating the exponent in $b=1$, we have that the last inequality holds (choosing b sufficiently near 1) for every $c>\frac{1}{1-2\alpha+\varepsilon}$ (note that $\frac{1}{1-2\alpha+\varepsilon}<\frac{3-2(\alpha+\varepsilon)}{2(1-2\alpha -\varepsilon)} $ for every $\alpha <\sfrac{1}{2}$ and $\varepsilon$ sufficiently small). We can then pick $a>1$ sufficiently large so that, by 
\eqref{e:Rq+1} and \eqref{e:Rq+1_adv}, the inequalities \eqref{e:R_C0_iter}, \eqref{e:R_C1_iter} and \eqref{e:R_Dt_iter}  hold for $\mathring{R}_{q+1}$. Note that in all these requirements, the energy profile $e$ is not playing any role.

\medskip

\noindent{\bf Estimates on $v_{q+1}-v_q$.} By \eqref{e:W_est_0}, Lemma \ref{l:ugly_lemma} and
\eqref{e:conditions_lambdamu_2} we conclude,  for $a$ sufficiently large, 
\begin{align}
\|v_{q+1} - v_q\|_0 &\leq \|w_o\|_0 + \|w_c\|_0 \leq \delta_{q+1}^{\sfrac{1}{2}} \left(\frac{M}{2} +
\lambda_{q+1}^{-\beta}\right)\, ,\\
\|v_{q+1} - v_q\|_1 &\leq \|w_o\|_1 + \|w_c\|_1 \leq \delta_{q+1}^{\sfrac{1}{2}} \lambda_{q+1} \left(\frac{M}{2} +
\lambda_{q+1}^{-\beta}\right)\, .
\end{align}
Since $\lambda_{q+1} \geq \lambda_1 \geq a^{cb^2}\geq 1$ and $M\geq 2$, we conclude
\eqref{e:v_C0_iter} and \eqref{e:v_C1_iter}.

\medskip

\noindent{\bf Estimate on the energy.} Recall Lemma \ref{l:energy}
and observe that, by \eqref{e:conditions_lambdamu_2}, 
$\frac{\delta_{q+1}^{\sfrac12}\delta_q^{\sfrac12}\lambda_q}{\lambda_{q+1}}\leq 
\frac{\delta_{q+1}^{\sfrac12}\mu}{\lambda_{q+1}}$. So the right hand side of \eqref{e:energy} is smaller than
\[
\frac{\|e'\|_0}{\mu} +C\delta_{q+1}^{\sfrac34}\delta_q^{\sfrac14}\lambda_q^{\sfrac12}\lambda_{q+1}^{-\sfrac12}+C \frac{\delta_{q}^{\alpha +\varepsilon}\lambda_{q}^{2\alpha+\varepsilon}}{\mu} .
\]
The term $C\delta_{q+1}^{\sfrac34}\delta_q^{\sfrac14}\lambda_q^{\sfrac12}\lambda_{q+1}^{-\sfrac12} $ is the same (up to a factor $ \lambda_{q+1}^\varepsilon$) of the first term in the estimate for  $ \mathring{R}_{q+1}$ Thus, the argument used above also gives $ C\delta_{q+1}^{\sfrac34}\delta_q^{\sfrac14}\lambda_q^{\sfrac12}\lambda_{q+1}^{-\sfrac12} \leq \frac{\delta_{q+2}}{16} e(t)$. Moreover it turns out that, for $b$ sufficiently near $1$
\[
C \frac{\delta_{q}^{\alpha +\varepsilon}\lambda_{q}^{2\alpha+\varepsilon}}{\mu} \leq \frac{\delta_{q+2}}{16} e(t)
\]
for any $ c> \frac{3-2(\alpha+\varepsilon)}{2(1-2\alpha -\varepsilon)}>\frac{3-2\alpha}{2(1-2\alpha)}$.
Regarding the last term $\frac{\|e'\|_0}{\mu} $ we have to require that (using the definition of $\mu$)
\begin{equation}
C \delta_{q+1}^{-\sfrac{1}{4}}\delta_q^{-\sfrac{1}{4}}\lambda_q^{-\sfrac12}\lambda_{q+1}^{-\sfrac12} \delta_{q+2}^{-1}\leq \frac{1}{\| e'\|_0}.
\end{equation}
The last inequality surely holds if we take the constant $a$ sufficiently large, more precisely
\begin{equation}
a\geq C \|e'\|_0
\end{equation}
is certainly sufficient.

\medskip

\noindent{\bf Estimates on $p_{q+1} - p_q$.}
From the definition of $p_{q+1}$ in \eqref{e:def_p_1} we deduce
\begin{align*}
\|p_{q+1}-p_q\|_0 &\leq \frac{1}{2} (\|w_o\|_0 + \|w_c\|_0)^2 + C\ell \|v_q\|_1 \|w\|_0\, .
\end{align*}
As already argued in the estimate for \eqref{e:v_C0_iter}, $\|w_o\| + \|w_c\|\leq M \delta^{\sfrac12}_q$.
Moreover $C \ell \|v_{q}\|_1\|w\|_0 \leq C M \delta_{q+1}^{\sfrac12} \delta_q^{\sfrac12} \lambda_q \ell$,
which is smaller than the right hand side of \eqref{e:allR}. Having already argued that such
quantity is smaller than $\eta \delta_{q+2}$ we can obviously bound $C \ell \|v_q\|_1 \|w\|_0$
with $\frac{M^2}{2} \delta_{q+1}$. This shows \eqref{e:p_C0_iter}. Moreover,
differentiating \eqref{e:def_p_1} we achieve the
bound
\begin{align*}
\|p_{q+1}-p_q\|_1 &\leq (\|w_o\|_1 + \|w_c\|_1) (\|w_o\|_0 + \|w_c\|_0) + C \delta_{q+1}^{\sfrac{1}{2}}
\delta_q^{\sfrac{1}{2}} \lambda_q \lambda_{q+1} \ell 
\end{align*}
and arguing as above we conclude \eqref{e:p_C1_iter}.

\smallskip

\noindent{\bf Estimates \eqref{e:t_derivatives}.} Here we can use the obvious identity 
$\partial_tw_q=D_tw_q - (v_q)_\ell\cdot\nabla w_q$ together with Lemmas \ref{l:ugly_lemma} and \ref{l:ugly_lemma_2}
to obtain $\|\partial_t v_{q+1} - \partial_t v_q\|_0 \leq C \delta_{q+1}^{\sfrac{1}{2}} \lambda_{q+1}$
Then, using \eqref{e:summabilities}, we conclude $\|\partial_t v_q\|_0 \leq C \delta_q^{\sfrac{1}{2}} \lambda_q$. 

To handle $\partial_t p_{q+1} - \partial_t p_q$ observe first that, by our construction,
\begin{align}
&\|\partial_t (p_{q+1} - p_q)\|_0 \leq (\|w_c \|_0 + \|w_o\|_0) (\|\partial_t w_c\|_0+ \|\partial_t w_o\|_0)\nonumber\\
&\qquad\qquad\qquad\qquad + 2 \|w\|_0\|\partial_t v_q\|_0 
+ \ell\|v_q\|_1\|\partial_t w\|_0\, .\nonumber
\end{align}
As above, we can derive the estimates 
$\|\partial_t w_o\|_0 +\|\partial_t w_c\|_0\leq C \delta_{q+1}^{\sfrac{1}{2}} \lambda_{q+1}$ from Lemmas \ref{l:ugly_lemma} and \ref{l:ugly_lemma_2}.
Hence
\begin{align}
\|\partial_t (p_{q+1} - p_q)\|_0 &\leq C \delta_{q+1} \lambda_{q+1} + C \delta_{q+1}^{\sfrac{1}{2}} \delta_q^{\sfrac{1}{2}} \lambda_q + C \delta_q^{\sfrac{1}{2}} \lambda_q \ell \delta_{q+1}^{\sfrac{1}{2}} \lambda_{q+1}\, .
\end{align}
Since $\ell \leq \lambda_q^{-1}$ and $\delta_q^{\sfrac{1}{2}}\lambda_q\leq \delta_{q+1}^{\sfrac{1}{2}} \lambda_{q+1}$, the desired inequality follows.
This concludes the proof.

\subsection{Proof of Proposition \ref{p:tecnica}} First, we set up the same iteration as above and for each energy profile we create a corresponding sequence $(v_{e,q}, p_{e,q}, \mathring{R}_{e,q})$ so that $(v_{e,q}, p_{e,q})$ converges uniformly to $(v_e, p_e)$. However, we choose the $\delta_q$ and $\lambda_q$ ``universally'' for all energy profiles $e\in \mathscr{E}$: it suffices to notice that we just need to replace the $\|e\|_{C^1}$ and $\|e\|_{C^2}$ in
\eqref{e:condizioni_su_abc_2} with $E_1$ and $E_2$. In particular, we fix the same $b$ and $c$ for every $e\in \mathscr{E}$ and we choose $a$ as
\begin{equation}\label{def_e}
a = \max\left\{a_0 (b,c), C_0E_1, C_0 E_2^{\frac{1}{(2c-1)b-1}}\right\}
\end{equation}
We then choose the starting triple $(v_{e,0}, p_{e,0}, \mathring{R}_{e,0})$ as in the proof of Lemma \ref{l:partenza} but where we define the parameter $\bar\lambda$ as
\begin{equation}\label{e:barlambda2}
\bar\lambda = C_0 \max \left\{ a^{\frac{b}{1-2\alpha}}, E_1 a^b, E_2 a^{- (c-1) b + 1/2}\right\}
\end{equation}
rather than by \eqref{e:barlambda}. In particular this means that, since $e (0)$ and $e' (0)$ are independent of $e\in \mathscr{E}$, the velocity $v_{e,0}$ and the Reynolds stress $\mathring{R}_{e,0}$ have the same initial value $v_{e, 0} (\cdot, 0)$ and $\mathring{R}_{e,0} (\cdot, 0)$ for every $e\in \mathscr{E}$. Following then the inductive construction of the triple $(v_{e,q}, p_{e,q}, \mathring{R}_{e,q})$, it is straightforward to conclude that each $v_{e,q+1} (\cdot, 0)$ and $\mathring{R}_{e,q+1} (\cdot, 0)$ depend only upon the $v_{q,e} (\cdot, 0)$, $\mathring{R}_{q,e} (\cdot, 0)$ and $e (0)$, hence concluding that such values are also independent of the chosen $e\in \mathscr{E}$. Passing to the limit such information we conclude that $v_e (\cdot, 0)$ is independent of $e\in \mathscr{E}$, namely each $v_e$ takes the same initial data.

\medskip

Coming to \eqref{e:chiave}, assume $\alpha < \frac{1}{5}$ and first of all observe that we can use the same argument as in the proof of Theorem \ref{t:main} to conclude that
$\|v-v_0\|_{C^{\alpha+\varepsilon}} \leq C_0$ provided $\varepsilon$ is sufficiently small. Moreover, the argument of Lemma \ref{l:partenza} gives the corresponding estimate
\[
\|v_0\|_{C^1} \leq \max \left\{ a^{\frac{b}{1-2\alpha}}, a^b E_1, E_2^{\frac{cb - 1/2}{(2c-1)b -1}}\right\}
\]
We can thus estimate
\begin{align*}
\|v\|_{C^{\alpha+\varepsilon}} &\leq C_0 + \|v_0\|_{C^{\alpha+\varepsilon}} \leq C_0 + C_0 \|v_0\|_1^{\alpha+\varepsilon}\\
& \leq C_0 \left(\max \left\{ a^{\frac{b}{1-2\alpha}}, a^b E_1, E_2^{\frac{cb - 1/2}{(2c-1)b -1}}\right\}\right)^{\alpha+\varepsilon}\, ,
\end{align*}
where $C_0$ is a geometric constant.
Now, observe first that 
\[
\lim_{(c,b) \to ({5}/{2}, 1)} \frac{cb-1/2}{(2c-1)b-1} = \frac{2}{3}\, ,
\]
Thus choosing $c-\frac{5}{2}$ and $b-1$ suitably small we can achieve
\[
\|v\|_{C^{\alpha+\varepsilon}} \leq C_0 \max \left\{a^{\frac{(\alpha+\varepsilon) b}{1-2\alpha}}, 
a^{(\alpha+\varepsilon)b} E_1^{\alpha+\varepsilon}, E_2^{\frac{2\alpha+ 4\varepsilon}{3}}\right\}\, .
\]
Note next that
\[
\lim_{\varepsilon \to 0} \frac{\alpha+\varepsilon}{1-2\alpha}=  \frac{\alpha}{1-2\alpha}
\]
whereas 
\[
\lim_{(\varepsilon,c,b) \to (0, {5}/{2}, 1)} \frac{(\alpha+\varepsilon) b}{(c-1)b-1/2} = \alpha\, 
\] 
Moreover $2\alpha > \frac{\alpha}{1-2\alpha}$, because $\alpha < \frac{1}{5}$.

Similarly,
\[
\lim_{(\varepsilon,c,b) \to (0, {5}/{2}, 1)}\frac{(\alpha+\varepsilon)b}{(1-2\alpha) ((2c-1)b -1)}
= \frac{\alpha}{3(1-2\alpha)} \leq \frac{2\alpha+4\varepsilon}{3}\, .
\]
Thus we conclude
\[
\|v\|_{C^{\alpha+\varepsilon}} \leq C_0 \max 
\left\{a^{\alpha+2 \varepsilon}E_1^{\alpha+\varepsilon}, E_2^{\frac{2\alpha+ 4\varepsilon}{3}}\right\}\, ,
\]
provided $\varepsilon$, $c-\frac{5}{2}$ and $b-1$ are sufficiently small. Inserting now the choice of $a$ in this last inequality (and again choosing the parameters $b-1, c -\frac{5}{2}$ and $\varepsilon$ appropriately small) we conclude
\begin{align*}
\|v\|_{C^{\alpha+\varepsilon}} &\leq C (\alpha, \varepsilon)  \max \left\{E_1^{2\alpha+3\varepsilon},
E_2^{\frac{\alpha+2\varepsilon}{3}}E_1^{\alpha+\varepsilon}, E_2^{\frac{2\alpha +4\varepsilon}{3}} \right\}\\
&\leq C (\alpha, \varepsilon) \max \left\{E_1^{2\alpha+3\varepsilon},  E_2^{\frac{2\alpha +4\varepsilon}{3}} \right\}\, .
\end{align*}
This shows \eqref{e:chiave} and completes the proof.

\appendix

\section{Proof of Theorem \ref{t:leray}}

We first consider the operator $\mathbf{P}_K : L^2 (\T^3) \to L^2 (\T^3)$ which truncates the Fourier series of each function $f\in L^2 (\T^3)$:
\[
\mathbf{P}_K (f) (x) = \sum_{|k|\leq K} \hat{f}_k e^{ik\cdot x}\, 
\]
and we extend it to vector functions by applying it to each component. Observe that the operator commutes with the derivatives. 

We then consider the regularized Cauchy problem:
\begin{equation}\label{NS_reg}
\left\{\begin{array}{l}
\partial_t w + \div \mathbf{P}_K (w\otimes w) + \nabla q+(-\Delta)^\alpha w =0\\   \\
\div w = 0 \\ \\
w (\cdot, 0) = \mathbf{P}_K (\overline v)\, .
\end{array}\right.
\end{equation}
The latter reduces to a system of ordinary differential equations for the Fourier coefficients of the solution
\[
w = \sum_{|k|\leq K} \hat{w}_k (t) e^{ik\cdot x}\, ,
\]
ensuring local well--posedness. On the other hand, if we scalar multiply the first equation by $w$ and use Plancherel's theorem, we easily see that
\[
\frac{d}{dt} \int_{\T^3} |w|^2 (x,t)\, dx = - 2 \int_{\T^3} |(-\Delta)^{\sfrac{\alpha}{2}} w|^2 (x,t)\, dx\, , 
\]
proving therefore that any solution stays bounded in $L^2$ in its interval $[0, T[$ of existence. This, by a standard continuation argument, proves that the system of ODEs for $\hat{w}_k (t)$ has a global solution on $\R^+$, namely that \eqref{NS_reg} is globally solvable. We let $w_K$ be such solution and observe therefore that 
\begin{align*}
\frac{1}{2} \int_{\T^3} |w_K|^2 (x,t)\, dt + \int_0^t \int_{\T^3}  |(-\Delta)^{\sfrac{\alpha}{2}} w_K| (x,s)\, dx\, ds &= \frac{1}{2} \int_{\T^3} |\mathbf{P}_K (\overline v) (x)|^2\, dx\\
&\leq \frac{1}{2} \int_{\T^3} |\overline v|^2 (x)\, dx\, .
\end{align*}
Let $K\in \mathbb N$: the sequence $\{w_K\}_{K\in \mathbb N}$ is thus bounded in $L^2 (\T^3 \times [0,T])$ for every $T$ and we can extract a subsequence, not relabeled, so that $w_K \rightharpoonup v$ in $L^2 (\T^3\times [0,T])$. With a standard diagonal argument we can then assume that such convergence takes place on every $\T^3\times [0,T]$. We now wish to show that in fact the sequence converges locally strongly, which would show that $v$ is a Leray solution.

\medskip

Since we have a uniform estimate for $w_K$ in $L^2 (\R^+, H^{\alpha} (\T^3))$ and $H^{\alpha} (\T^3)$ embeds compactly in $L^2 (\T^3)$, the proof follows a classical Aubin--Lions type argument. First of all, by Sobolev embeddings, $\|w_K\|_{L^2(\R^+,L^\beta(\T^3))}\le C$ for some $\beta >2$. Hence, by interpolation with the $L^\infty (\R^+, L^2)$ bound, we have also
$\|w_K\|_{L^\gamma (\T^3\times [0,\infty[)}\leq C$ for some exponent $\gamma>2$ 

 Let us fix $T>0$ and define
\[
A_{K,J}:=\|w_K-w_J\|_{L^2(\T^3 \times [0,T])}\,.
\]
Let $\varepsilon > 0$ be given. 
We want to show that $\exists\,N\in\N$ sucht that $A_{K,J}<\varepsilon$ for every $K,J\ge N$. Fix a standard mollifier $\varphi_\delta$ in the variable $x$ and observe that
\[
\|w_K(\cdot,t)-w_K\ast\varphi_\delta(\cdot,t)\|_{L^2}\le C \delta^{\alpha} \|w_k(\cdot,t)\|_{H^{\alpha}}\,.
\]
So, for $\delta$ sufficiently small, we have that
\[
\|w_K\ast\varphi_\delta-w_K\|_{L^2(K)}<\frac{\varepsilon}{3}\qquad\forall\,K\in\N\,.
\]
Next, observe that
\[
\|\mathbf{P}_K (w_K\otimes w_K)\|_{L^{\sfrac{\gamma}{2}}} \leq C
\]
and, using the equation
\[
- \Delta q_K = \div \div (\mathbf{P}_K (w_K\otimes w_K))
\]
for the corresponding pressure $q_K$
and Calderon--Zygmund estimates,
\[
\|q_K\|_{L^{\sfrac{\gamma}{2}}} \leq C\, .
\]
Thus, mollifying the equation for $w_K$ we find
\[
\partial_t w_K*\varphi_\delta = - \sum_{i=1}^3 f_{i,K} * \partial_{x_i} \varphi_\delta - w_K * (-\Delta)^{\alpha} \varphi_\delta\, ,
\]
where the functions $f_{i, K}$ enjoy a uniform $L^{\sfrac{\gamma}{2}}$ bound. 
Using the estimate $\|\zeta * \varphi_\delta \|_{W^{1, \infty} (\T^3)} \leq C (\delta) \|\zeta\|_{L^{\sfrac{\gamma}{2}}}$ for each time slice, we easily conclude a bound of the form
\[
\int_0^T \|\partial_t w_K * \varphi_\delta (\cdot, t)\|^{\sfrac{\gamma}{2}}_{W^{1, \infty}}\, dt \leq C (\delta)\, ,
\]
where $C (\delta)$ is a constant depending upon $\delta$ but independent of $K$. 

So we can regard $[0,T]\ni t \mapsto w_K*\varphi_\delta (\cdot, t)$ as a sequence of equicontinuous and equibounded curves taking values in $W^{1, \infty} (\T^3)$. Let $B_R$ be a (closed) ball of $W^{1,\infty} (\T^3)$ so that the images of $w_K*\varphi_\delta$ are all contained inside it. If we endow $B_R$ with the $\|\cdot\|_\infty$ norm, then we have a compact metric space $X$. Hence we can regard $[0,T]\ni t \mapsto w_K*\varphi_\delta (\cdot, t)$ as an equicontinuous and equibounded sequence in the compact metric space $X$. By the Ascoli--Arzel\`a theorem the sequence is then precompact. Since the limit is unique (namely $v * \varphi_\delta$), we can conclude that the sequence $w_K*\varphi_\delta$ converges uniformly on $\T^3 \times [0,T]$. 

Thus there exists $N$ large enough such that 
\[
\|w_K\ast\varphi_\delta-w_J\ast\varphi_\delta\|_{L^2(\T^3 \times [0,T])}<\frac{\varepsilon}{3} \qquad \mbox{for all $K,J\geq N$.}
\]
Therefore, for $J,K\geq N$ we have
\begin{align*}
\|w_K-w_J\|_{L^2 (\T^3\times [0,T])} &\leq \|w_K - w_K\ast \varphi_\delta\|_{L^2 (\T^3 \times [0,T])} + 
\|w_K\ast \varphi_\delta - w_J \ast \varphi_\delta\|_{L^2 (\T^3\times [0,T])}\\
&\quad  + \|w_J - w_J\ast \varphi_\delta\|_{L^2 (\T^3\times [0,T])}
< \varepsilon\, .
\end{align*}
This completes the proof of the strong convergence of $w_K$ and hence the proof of Theorem \ref{t:leray}. 

\section{H\"older spaces}\label{s:hoelder}

In the following $m=0,1,2,\dots$, $\alpha\in (0,1)$, and $\beta$ is a multi-index. We introduce the usual (spatial) 
H\"older norms as follows.
First of all, the supremum norm is denoted by $\|f\|_0:=\sup_{\T^3\times [0,1]}|f|$. We define the H\"older seminorms 
as
\begin{equation*}
\begin{split}
[f]_{m}&=\max_{|\beta|=m}\|D^{\beta}f\|_0\, ,\\
[f]_{m+\alpha} &= \max_{|\beta|=m}\sup_{x\neq y, t}\frac{|D^{\beta}f(x, t)-D^{\beta}f(y, t)|}{|x-y|^{\alpha}}\, ,
\end{split}
\end{equation*}
where $D^\beta$ are {\em space derivatives} only.
The H\"older norms are then given by
\begin{eqnarray*}
\|f\|_{m}&=&\sum_{j=0}^m[f]_j\\
\|f\|_{m+\alpha}&=&\|f\|_m+[f]_{m+\alpha}.
\end{eqnarray*}
Moreover, we will write $[f (t)]_\alpha$ and $\|f (t)\|_\alpha$ when the time $t$ is fixed and the
norms are computed for the restriction of $f$ to the $t$-time slice.

Recall the following elementary inequalities:
\begin{equation}\label{e:Holderinterpolation}
[f]_{s}\leq C\bigl(\varepsilon^{r-s}[f]_{r}+\varepsilon^{-s}\|f\|_0\bigr)
\end{equation}
for $r\geq s\geq 0$, $\eps>0$, and 
\begin{equation}\label{e:Holderproduct}
[fg]_{r}\leq C\bigl([f]_r\|g\|_0+\|f\|_0[g]_r\bigr)
\end{equation}
for any $1\geq r\geq 0$. From \eqref{e:Holderinterpolation} with $\eps=\|f\|_0^{\frac1r}[f]_r^{-\frac1r}$ we obtain the 
standard interpolation inequalities
\begin{equation}\label{e:Holderinterpolation2}
[f]_{s}\leq C\|f\|_0^{1-\frac{s}{r}}[f]_{r}^{\frac{s}{r}}.
\end{equation}

Next we collect two classical estimates on the H\"older norms of compositions. These are also standard, for instance
in applications of the Nash-Moser iteration technique.

\begin{proposition}\label{p:chain}
Let $\Psi: \Omega \to \mathbb R$ and $u: \R^n \to \Omega$ be two smooth functions, with $\Omega\subset \R^N$. 
Then, for every $m\in \mathbb N \setminus \{0\}$ there is a constant $C$ (depending only on $m$,
$N$, $n$) such that
\begin{align}
\left[\Psi\circ u\right]_m &\leq C ([\Psi]_1 [u]_{m}+\|D\Psi\|_{m-1} \|u\|_0^{m-1} [u]_{m})\label{e:chain0}\\
\left[\Psi\circ u\right]_m &\leq C ([\Psi]_1 [u]_{m}+\|D\Psi\|_{m-1} [u]_1^{m} )\, .
\label{e:chain1}
\end{align} 
\end{proposition}

\section{Estimates on the fractional laplacian}

For a proof of the following theorem we refer to \cite[Theorem 1.4]{Roncal}

 \begin{theorem}{(Interaction with Holder spaces)}\label{lapla.holder}
Let $ 0< \alpha < \dfrac{1}{2}$. If $f \in C^{0,2 \alpha + \varepsilon} $, for some $\varepsilon >0$ such that $ 0<2 \alpha + \varepsilon \leq 1$, then $ (-\Delta)^\alpha f$ is a continous function and

\[ 
 \| (-\Delta)^\alpha f\|_0 \leq C(\varepsilon) [f]_{2 \alpha + \varepsilon}.
\]

 \end{theorem}

\begin{corollary}\label{vlaplv}
Let $\alpha \in ]0,1[$, $\varepsilon >0$ be such that $0<\alpha+\varepsilon\leq 1$, and let $f:\mathbb{T}^3 \rightarrow \mathbb{R}^3$ as before. There exist a constant $C=C(\varepsilon)>0$ such that
\begin{equation}
\int_{\mathbb{T}^3} |(-\Delta)^{\sfrac{\alpha}{2}} f|^2 (x) dx \leq C [f]_{\alpha+\varepsilon}^2
\qquad \forall f \in C^{\alpha+\varepsilon} (\T^3)\, .
\end{equation}
\end{corollary}

\section{Estimates for transport equations}\label{s:transport_equation}

In this section we recall some well known results regarding smooth solutions of
the \emph{transport equation}:
\begin{equation}\label{e:transport}
\left\{\begin{array}{l}
\partial_t f + v\cdot  \nabla f =g,\\ 
f|_{t_0}=f_0,
\end{array}\right.
\end{equation}
where $v=v(t,x)$ is a given smooth vector field. 
We denote the advective derivative $\partial_t+v\cdot \nabla$ by $D_t$. We will consider solutions
on the entire space $\R^3$ and treat solutions on the torus simply as periodic solution in $\R^3$. 

\begin{proposition}\label{p:transport_derivatives}
Assume $t>t_0$. Any solution $f$ of \eqref{e:transport} satisfies
\begin{align}
\|f (t)\|_0 &\leq \|f_0\|_0 + \int_{t_0}^t \|g (\tau)\|_0\, d\tau\,,\label{e:max_prin}\\
[f(t)]_1 &\leq [f_0]_1e^{(t-t_0)[v]_1} + \int_{t_0}^t e^{(t-\tau)[v]_1} [g (\tau)]_1\, d\tau\,,\label{e:trans_est_0}
\end{align}
and, more generally, for any $N\geq 2$ there exists a constant $C=C_N$ so that
\begin{align}
[f (t)]_N & \leq \Bigl([ f_0]_N + C(t-t_0)[v]_N[f_0]_1\Bigr)e^{C(t-t_0)[v]_1}+\nonumber\\
&\qquad +\int_{t_0}^t e^{C(t-\tau)[v]_1}\Bigl([g (\tau)]_N + C (t-\tau ) [v ]_N [g (\tau)]_{1}\Bigr)\,d\tau.
\label{e:trans_est_1}
\end{align}
Define $\Phi (t, \cdot)$ to be the inverse of the flux $X$ of $v$ starting at time $t_0$ as the identity
(i.e. $\frac{d}{dt} X = v (X,t)$ and $X (x, t_0 )=x$). Under the same assumptions as above:
\begin{align}
\norm{D\Phi (t) -\Id}_0&\leq e^{(t-t_0)[v]_1}-1\,,  \label{e:Dphi_near_id}\\
[\Phi (t)]_N &\leq C(t-t_0)[v]_Ne^{C(t-t_0)[v]_1} \qquad \forall N \geq 2.\label{e:Dphi_N}
\end{align}
\end{proposition}

The proof can be found in \cite{BDIS}.

\section{Constantin-E-Titi commutator estimate}\label{s:CET}

We recall here the quadratic commutator estimate from \cite{CET} 
(cf. also with \cite[Lemma 1]{CDSz}):

\begin{proposition}\label{p:CET}
Let $f,g\in C^{\infty}(\T^3\times\T)$ and $\psi$ be the mollifier of Section \ref{s:mollifier}. For any $r\geq 0$ we have the estimate
\[
\Bigl\|(f*\psi_\ell)( g*\psi_\ell)-(fg)*\psi_\ell\Bigr\|_r\leq C\ell^{2-r}\|f\|_1\|g\|_1\, ,
\]
where the constant $C$ depends only on $r$.
\end{proposition}

\section{Schauder Estimates}
\label{s:schauder}

We recall the following consequences of the classical Schauder estimates  (cf.\ 
\cite[Proposition 5.1]{DS-Inv}). 

\begin{proposition}\label{p:GT}
For any $\alpha\in (0,1)$ and any $m\in \N$ there exists a constant $C (\alpha, m)$ with the following properties.
If $\phi, \psi: \T^3\to \R$ are the unique solutions of
\[
\left\{\begin{array}{l}
\Delta \phi = f\\ \\
\fint \phi =0
\end{array}\right.
\qquad\qquad 
\left\{\begin{array}{l}
\Delta \psi = {\rm div}\, F\\ \\
\fint \psi =0
\end{array}\right. \, ,
\]
then
\begin{equation}\label{e:GT_laplace}
\|\phi\|_{m+2+\alpha} \leq C (m, \alpha) \|f\|_{m, \alpha}
\quad\mbox{and}\quad \|\psi\|_{m+1+\alpha} \leq C (m, \alpha) \|F\|_{m, \alpha}\, .
\end{equation}
Moreover we have the estimates
\begin{eqnarray}
&&\|\mathcal{R} v\|_{m+1+\alpha} \leq C (m,\alpha) \|v\|_{m+\alpha}\label{e:Schauder_R}\\
&&\|\mathcal{R} ({\rm div}\, A)\|_{m+\alpha}\leq C(m,\alpha) \|A\|_{m+\alpha}\label{e:Schauder_Rdiv}
\end{eqnarray}
\end{proposition}

\section{Stationary phase and commutator lemmas}\label{s:stationary}

Finally, we will need the following stationary phase lemma (for a proof see \cite{DS-Inv}) and a useful commutator estimate (for a proof see \cite{BDIS}). 

\begin{proposition}\label{p:stat_phase}
(i) Let $k\in\Z^3\setminus\{0\}$ and $\lambda\geq 1$ be fixed. 
For any $a\in C^{\infty}(\T^3)$ and $m\in\N$ we have
\begin{equation}\label{e:average}
\left|\int_{\T^3}a(x)e^{i\lambda k\cdot x}\,dx\right|\leq \frac{[a]_m}{\lambda^m}.
\end{equation}

(ii) Let $k\in\Z^3\setminus\{0\}$ be fixed. For a smooth vector field $a\in C^{\infty}(\T^3;\R^3)$ let 
$F(x):=a(x)e^{i\lambda k\cdot x}$. Then we have
\begin{equation}\label{e:R(F)}
\|\RR(F)\|_{\alpha}\leq \frac{C}{\lambda^{1-\alpha}}\|a\|_0+\frac{C}{\lambda^{m-\alpha}}[a]_m+\frac{C}{\lambda^m}[a]_{m+\alpha},
\end{equation}
where $C=C(\alpha,m)$.
\end{proposition}

\begin{proposition}\label{p:commutator}
Let $k\in\Z^3\setminus\{0\}$ be fixed. For any smooth vector field $a\in C^\infty  (\T^3;\R^3)$ and any smooth function $b$, if we set $F(x):=a(x)e^{i\lambda k\cdot x}$, we then have
\begin{align}
&\|[b, \mathcal{R}] (F)\|_\alpha \leq  \lambda^{\alpha-2} \|a\|_0 \|b\|_1
+ C \lambda^{\alpha-m} \left(\|a\|_{m-1+\alpha} \|b\|_{1+\alpha} + \|a\|_{\alpha} \|b\|_{m+\alpha}\right)\label{e:main_est_commutator}
\end{align}
where $C=C(\alpha,m)$.
\end{proposition}

\bibliographystyle{plain}
\bibliography{NS-alfa}
\end{document}